\documentclass[preprint,3p]{amsart}
\usepackage[latin1,utf8]{inputenc}

\usepackage{geometry}
\usepackage{amsmath, mathtools}
\usepackage{pdfpages}
\usepackage{amstext}
\usepackage{amsthm}
\usepackage{amssymb}
\usepackage{stmaryrd}
\usepackage{graphicx}
\usepackage{bbm}
\usepackage{comment}
\usepackage{faktor}
\graphicspath{{Pictures/}}

\makeatletter

\numberwithin{equation}{section}
\numberwithin{figure}{section}
\theoremstyle{plain}
\newtheorem{thm}{\protect\theoremname}[section]
  \theoremstyle{definition}
  \newtheorem{defn}[thm]{\protect\definitionname}
  \theoremstyle{remark}
  \newtheorem{rem}{\protect\remarkname}
  \theoremstyle{plain}
  \newtheorem{lem}[thm]{\protect\lemmaname}
  \theoremstyle{plain}
  \newtheorem{prop}[thm]{\protect\propositionname}
  \theoremstyle{plain}
  \newtheorem{cor}[thm]{\protect\corollaryname}
  \theoremstyle{plain}
\newtheorem{que}[thm]{\protect\questionname}
  \theoremstyle{plain}

\makeatother

  \providecommand{\definitionname}{Definition}
  \providecommand{\lemmaname}{Lemma}
  \providecommand{\propositionname}{Proposition}
  \providecommand{\remarkname}{Remark}
\providecommand{\theoremname}{Theorem}
\providecommand{\corollaryname}{Corollary}
\providecommand{\questionname}{Question}

\title [Oscillatory motions in the R3BP: A Functional Analytic Approach]{Oscillatory motions in the Restricted   3-Body Problem: A Functional Analytic Approach}

\author[J. Paradela]{Jaime Paradela}
\address[JP]{ Departament de Matem\`atiques, Universitat Polit\`ecnica de Catalunya, Diagonal 647, 08028 Barcelona, Spain}
\email{jaime.paradela@upc.edu}

\author[S. Terracini]{Susanna Terracini}
\address[ST]{Dipartimento di Matematica Giuseppe Peano, Universit\`{a} di Torino, Via Carlo Alberto 10, 10123, Torino, Italy}
\email{susanna.terracini@unito.it}


\makeatletter
\@namedef{subjclassname@2020}{\textup{2020} Mathematics Subject Classification}
\makeatother


\thanks{This work was partially done while the first author was visiting the Mathematics Department of the University of Turin and he thanks the department for their hospitality and pleasant working atmosphere.  This project has received funding from the European Research Council (ERC) under the European Union’s Horizon 2020 research and innovation programme (grant agreement No 757802).}

\begin{document}
\maketitle

\begin{abstract}
A fundamental question in Celestial Mechanics is to analyze the possible final motions of the Restricted $3$-body Problem, that is, to provide the qualitative description of its complete (i.e. defined for all time) orbits  as time goes to infinity. According to the classification given by Chazy back in 1922, a remarkable  possible  behaviour is that of oscillatory motions, where the motion $q$ of the massless body is unbounded but  returns infinitely often inside some bounded region:
\[
\limsup_{t\to\pm\infty} |q(t)|=\infty\qquad\qquad\text{and}\qquad\qquad \liminf_{t\to\pm\infty} |q(t)|<\infty.
\]
In contrast with the other possible final motions in Chazy's classification, oscillatory motions do not occur in the $2$-body Problem, while they do for larger numbers of bodies. A further point of interest is their appearance in connection with the existence of chaotic dynamics.

In this paper we introduce new tools to study the existence of oscillatory motions and prove that oscillatory motions exist in a particular configuration known as the Restricted Isosceles $3$-body Problem (RI3BP) for almost all values of the angular momentum.  Our method, which is global and not limited to nearly integrable settings, extends the previous results  \cite{guardia2021symbolic} by blending variational and geometric techniques with tools from nonlinear analysis such as topological degree theory. To the best of our knowledge, the present work constitutes the first complete analytic proof of existence of oscillatory motions in a non perturbative regime.
\end{abstract}

\tableofcontents


\section{Introduction}

One of the oldest questions in Dynamical Systems is to understand the mechanisms driving the global dynamics of the $3$-body problem, which models the motion of three bodies interacting through Newtonian gravitational force. The  $3$-body Problem is called \textit{restricted}  if one of the bodies has mass zero and the other two have strictly positive masses.  In this limit problem, the massless body is affected by, but does not affect, the motion of the massive bodies. A fundamental question concerning the global dynamics of the Restricted $3$-body Problem is the study of its possible final motions, that is, the qualitative description of its complete (defined for all time) orbits as time goes to infinity. In 1922 Chazy gave a complete classification of the possible final motions of the Restricted $3$-body Problem \cite{MR1509241}. To describe them we denote by $q$ the position of the massless body in a Cartesian reference frame with origin at the center of mass of the primaries.

\begin{thm}[\cite{MR1509241}]\label{thm:Chazy}
Every solution of the Restricted $3$-body Problem defined for all (future) times belongs to one of the following classes
\begin{itemize}
\item B (bounded): $\sup_{t\geq 0} |q(t)|<\infty$.
\item P (parabolic) $|q(t)|\to \infty$ and $|\dot{q}(t)|\to 0$ as $t\to\infty$.
\item H (hyperbolic): $|q(t)|\to \infty$ and $|\dot{q}(t)|\to c>0$ as $t\to\infty$.
\item O (oscillatory) $\limsup_{t\to\infty} |q(t)|=\infty$ and $\liminf_{t\to\infty} |q(t)|<\infty$.
\end{itemize}
\end{thm}
Notice that this classification also applies for $t\to-\infty$. We distingish both cases adding a superindex $+$ or $-$ to each of the cases, e.g. $H^+$ and $H^-$.

Bounded, parabolic and hyperbolic motions also exist in the $2$-body Problem, and examples of each of these classes of motion in the Restricted $3$-body Problem were already known by Chazy. However, the existence of oscillatory motions in the Restricted $3$-body Problem was an open question for a long time. Their existence was first established by Sitnikov in a particular configuration of the Restricted $3$-body Problem nowadays known as the Sitnikov problem. 

\subsection{The Moser approach to the existence of oscillatory motions}\label{sec:Moserapproach}
After Sitnikov's work, Moser gave a new proof of the existence of oscillatory motions in the Sitnikov problem \cite{MR1829194}. His approach makes use of tools from the geometric theory of dynamical systems, in particular, hyperbolic dynamics. More concretely, Moser considered an invariant periodic orbit ``at infinity" (see Section \ref{sec:isoscelesproblem}) which is degenerate (the linearized vector field vanishes) but posseses stable and unstable invariant manifolds. Then, he proved that its stable and unstable manifolds intersect transversally. Close to this intersection, he built a section $\Sigma$ transverse to the flow and established the existence of a (non trivial) locally maximal hyperbolic set $\mathcal{X}$ for the Poincar\'{e} map $\Phi_\Sigma$ induced on $\Sigma$. The dynamics of $\Phi_\Sigma$ restricted to $\mathcal{X}\subset \Sigma$ is moreover conjugated to the shift
\[
\sigma:\mathbb{N}^\mathbb{Z}\to\mathbb{N}^\mathbb{Z}\qquad\qquad (\sigma\omega)_k= \omega_{k+1}
\]
acting on the space of infinite sequences. Namely, $\mathcal{X}$ is a horseshoe with "infinitely many legs" for the map $\Phi_\Sigma$. By construction,  sequences $\omega=(\dots,\omega_{-n}, \omega_{-n+1}\dots \omega_0,\dots \omega_{n-1}\dots \omega_n\dots)\in\mathbb{N}^{\mathbb{Z}}$ for which $\limsup_{n\to \infty}\omega_n$ (respectively $\limsup_{n\to- \infty}\omega_n$) correspond to complete motions of the Sitnikov problem which are oscillatory in the future (in the past).

Moser's ideas have been very influential. In \cite{llibre1980oscillatory} Sim\'{o} and Llibre implemented Moser's approach in the Restricted Circular $3$-body Problem (RC3BP) in the region of the phase space with large Jacobi constant provided the  values of the ratio between the masses of the massive bodies is small enough. Their result was later extended by Xia \cite{MR1278373} and closed by Guardia, Mart\'{i}n and Seara in \cite{MR3455155} where oscillatory motions for the RC3BP for all mass ratios are constructed in the region of the phase space with large Jacobi constant. The same result is obtained in \cite{capinski2021oscillatory} for low values of the Jacobi constant relying on a computer assisted proof. In \cite{MR3583476} and \cite{MR4173566}, the Moser approach is applied to the Restricted Elliptic $3$-body Problem and the Restricted 4 Body Problem respectively. For the $3$-body Problem, results in certain symmetric configurations (which reduce the dimension of the phase space) were obtained in \cite{MR0249754} and \cite{MR590002}. Another interesting result, which however holds for non generic choices of the 3 masses, is obtained in \cite{MR2350333}. In the recent preprint \cite{https://doi.org/10.48550/arxiv.2207.14351}, the first author together with  Guardia,  Mart\'{i}n and Seara,  has proved the existence of oscillatory motions in the planar $3$-body Problem (5 dimensional phase space after symplectic reductions) for all choices of the masses (except all equal) and large total angular momentum.

The first main ingredient in  Moser's strategy is  the detection of a transversal intersection between the invariant manifolds of the periodic orbit at infinity. Yet, checking the occurrence of this phenomenon in a physical model is rather problematic, and in general little can be said except for perturbations of integrable systems with a hyperbolic fixed point whose stable and unstable manifolds coincide along a homoclinic manifold. As far as the authors know, all the previous works concerning the existence of oscillatory motions in the $3$-body Problem (restricted or not) adopt a perturbative approach to prove the existence of transversal intersections between the stable and unstable manifolds of infinity. In some cases the perturbative regime is obtained by assuming that certain parameter related to the motion of the massive bodies (in general the ratio between the masses of the massive bodies or the eccentriticy of their orbit) is small, and others by working in a region of the phase space where the massless body is located far away from the primaries. The latter situation falls in what is usually called singular perturbation theory and (in general) needs a much more involved analysis than the former one, usually referred to as regular perturbation theory. 

The second key ingredient is the construction of a horseshoe close to the transversal intersections of the invariant manifolds. For the Sitnikov and Isosceles Restricted $3$-body Problem (which is introduced in Section \ref{sec:isoscelesproblem}) are non autonomous Hamiltonian systems with $1+1/2$ degrees of freedom ($3$ dimensional phase space), its dynamics can be reduced to the study of a two dimensional area preserving map in which the periodic orbit at infinity becomes a fixed point which, despite being degenerate, behaves as a hyperbolic fixed point. The same happens in the RC3BP after reducing by rotational symmetry and in certain symmetric configurations of the 3BP. In all of these problems Moser's ideas for constructing a horseshoe close to the transverse intersections between the invariant manifolds of the parabolic fixed point can be implemented directly. In the planar $3$-body Problem, the dynamics can be reduced to a 4 dimensional symplectic map and the parabolic fixed point becomes a 2 dimensional (degenerate) normally hyperbolic invariant manifold. Due to the existence of central directions the construction of the horseshoe in \cite{https://doi.org/10.48550/arxiv.2207.14351}  becomes much more involved. In \cite{MR2350333}, the author analyzes orbits which pass close to triple collision. In this setting, the close encounters with triple collision, produce stretching also in the central directions.

An approach different in nature from Moser's is developed by Galante and Kaloshin in \cite{MR2824484}. By making use of Aubry-Mather theory and semi-infinite regions of instability, the authors prove the existence of oscillatory orbits for the RC3BP with a realistic value of the mass ratio.

\subsection{The fundamental question in Celestial Mechanics}

Besides the question of its existence, is the question of their abundance. In the conference in honor of the 70th anniversary of Alexeev, Arnol'd posed the following question (cfr \cite{MR2824484}).

\begin{que}
Is the Lebesgue measure of the set of oscillatory motions positive?
\end{que}

This question was considered by Arnol'd to be the fundamental issue of Celestial Mechanics. It has been conjectured by Alexeev that the Lebesgue measure is zero. Neverthless, this conjecture remains wide open. The only partial results in this direction are due to Gorodetski and Kaloshin \cite{GorodetskiK12}. They consider the RC3BP and  the Sitnikov problem and prove that for both problems and a Baire generic subset of an open set of parameters (eccentricity in the Sitnikov problem and mass ratio in the RC3BP), the Hausdorff dimension of the set of oscillatory motions is maximal. 

\subsection{The Isosceles configuration of the Restricted $3$-body Problem}\label{sec:isoscelesproblem}

In the present work we consider a particular configuration of the Restricted $3$-body Problem known as the Restricted Isosceles $3$-body Problem. In this configuration, the two primaries have equal masses $m_0=m_1=1/2$ and move periodically on a degenerate ellipse of eccentricity one (a line), according to the Kepler laws for the motion of the $2$-body Problem. The massless particle moves on the plane perpendicular to the line along which the primaries move (see Figure \ref{fig:Isoscelesconfiguration}).

In the plane of motion of the massless  body we fix a Cartesian reference frame with origin at the point where the line along which the primaries move intersects the plane. Then, in Cartesian coordinates $(q,p,t)\in\mathbb{R}^4\times\mathbb{T}\setminus\{q=0\}$, the motion of the massless body is given by the Hamiltonian system
\[
H(q,p,t)=\frac{|p|^2}{2}-V_{\mathrm{cart}}(q,t)\qquad\qquad V_{\mathrm{cart}}(q,t)=\frac{1}{\sqrt{|q|^2+\rho^2(t)}}.
\]
where $\rho(t):\mathbb{T}\to[0,1/2]$ is a half of the distance between the primaries.

\begin{rem}
One can obtain an explicit expression of the function $\rho(t)$ after introducing the change of variables $t=u-\sin u$, commonly known as the Kepler equation. When expressed in terms of the new variable $u$ (which is the eccentric anomaly) we have $\rho(t(u))=(1-\cos u)/2$. Yet, our analysis does not require to have an explicit expression of the function $\rho(t)$,  so we work directly with the original variable $t$. 
\end{rem}

\begin{figure}\label{fig:Isoscelesconfiguration}
\centering
\includegraphics[scale=0.50]{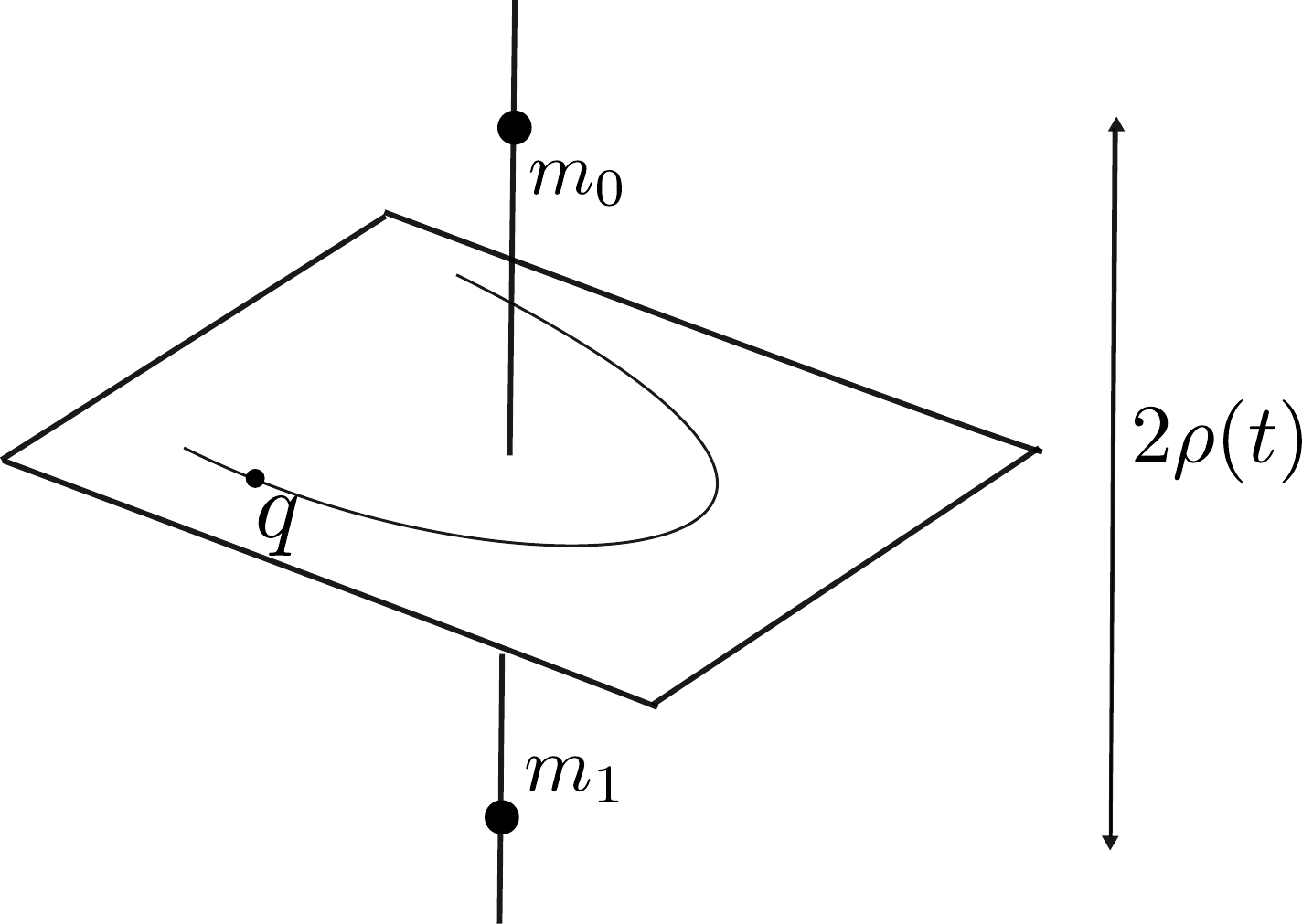}
\caption{Sketch of the motion in the Restricted Isosceles $3$-body Problem.}
\end{figure}

It will be convenient for our analysis to introduce polar coordinates  $(r,\alpha,t,y,G)\in\mathbb{R}_+\times\mathbb{T}^2\times\mathbb{R}^2$ where $q=(r\cos\alpha,r\sin\alpha)$ and  $(y,G)$ denote the conjugate momenta to $(r,\alpha)$. In polar coordinates, the Hamiltonian of the Restricted Isosceles $3$-body Problem reads
\begin{equation}\label{eq:Hamiltonian}
H(r,t,y,G)=\frac{y^2}{2}+\frac{G^2}{2r^2}-V(r,t)\qquad\qquad V(r,t)=\frac{1}{\sqrt{r^2+ \rho^2(t)}}.
\end{equation}
We inmediately notice that $G$ is a conserved quantity for the flow of \eqref{eq:Hamiltonian}. It is therefore natural to consider the one-parameter family of Hamiltonian systems 
\begin{equation}\label{eq:oneparameterHamiltonian}
H_G(r,t,y)=H(r,t,y,G)\qquad\qquad (r,t,y)\in\mathbb{R}_+\times\mathbb{T}\times\mathbb{R}.
\end{equation}
Since $\lim_{r\to\infty} V(r,t)=0$, for all $G\in\mathbb{R}$ the Hamiltonian  \eqref{eq:oneparameterHamiltonian} posses a periodic orbit at infinity
\begin{equation}\label{eq:periodicorbitinfty}
\gamma_{\infty}=\{r=\infty,\ y=0\}\subset \mathbb{R}_+\times\mathbb{T}\times\mathbb{R}.
\end{equation}
In \cite{guardia2021symbolic}, the first author together with M. Guardia, T. Seara and C.Vidal, proved the following result.
\begin{thm}[\cite{guardia2021symbolic}]\label{thm:oscillatoryjde}
Consider the Hamiltonian system $H_G$ defined in \eqref{eq:oneparameterHamiltonian}. Denote by $X^+$ (respectively $Y^-$) either $H^+,P^+,B^+$ or $OS^+$ (respectively $H^-,P^-,B^-$ or $OS^-$) according to Chazy's classification in Theorem \ref{thm:Chazy}. Then, there exists $G_*\gg1$ such that for all  $G\in\mathbb{R}$ such that $|G|\geq G_*$, the Hamiltonian system $H_G$ satisfies
\[
X^+\cap Y^-\neq \emptyset
\]
for all possible combinations of $X^+$ and $Y^-$.
\end{thm}

Theorem \ref{thm:oscillatoryjde} is proved by exploiting the fact that for $G$ large enough, in a suitable region of the phase space, the Hamiltonian $H_G$  can be studied as a perturbation of the (integrable) $2$-body Problem. This allowed the authors to prove that the periodic orbit $\gamma_\infty$ posses global stable and unstable invariant manifolds which intersect transversally (see Theorem \ref{thm:transversalityjde}). As a corollary of  this result, a rather straightforward implementation of Moser's ideas shows the truth of Theorem \ref{thm:oscillatoryjde}.\\

The following is the first main result of the present work.

\begin{thm}\label{thm:mainthmfinalmotions}
Consider the Hamiltonian system $H_G$ defined in \eqref{eq:oneparameterHamiltonian}. Denote by $X^+$ (respectively $Y^-$) either $H^+,P^+,B^+$ or $OS^+$ (respectively $H^-,P^-,B^-$ or $OS^-$) according to Chazy's classification in Theorem \ref{thm:Chazy}. Then, for almost all  $G\in\mathbb{R}$ the Hamiltonian system $H_G$ satisfies
\[
X^+\cap Y^-\neq \emptyset
\]
for all possible combinations of $X^+$ and $Y^-$.
\end{thm}

To the best of our knowledge, Theorem \ref{thm:mainthmfinalmotions} is the first complete analytic proof of the existence of oscillatory motions relying upon a global analytical approach rather than on perturbative techniques. Some interesting related works, where the existence of  oscillatory motions is obtained in a setting which is not close to integrable, are \cite{MR2350333} and \cite{capinski2021oscillatory}. While in  \cite{MR2350333} the author shows the existence of oscillatory motions in the $3$-body Problem close to triple collision (small values of the total angular momentum), in \cite{capinski2021oscillatory} the authors obtain a computer assisted proof of the existence of oscillatory motions in the Restricted Circular $3$-body Problem for small values of the Jacobi constant.\\

Theorem \ref{thm:mainthmfinalmotions} is indeed obtained as a consequence of the following result.

\begin{thm}\label{thm:maintheorem}
Let $\{l_j\}\subset \mathbb{Z}$ be an increasing sequence and define the time intervals $I_j=[(l_j-l_{j-1})/2,\ (l_{j+1}-l_j)/2]$. Then, for almost all $G\in\mathbb{R}$, all $\varepsilon> 0$ and all $R$ sufficiently large, there exists an orbit $r_{\mathrm{h}}(s):\mathbb{R}\to \mathbb{R}_+$  of \eqref{eq:oneparameterHamiltonian} homoclinic to $\gamma_\infty$ and a constant $L>0$ such that  if the sequence $\{l_j\}\subset \mathbb{Z}$ satisfies $l_{j+1}-l_j\geq L$, then, for any sequence  $\sigma=\{\sigma_j\}\subset \{0,1\}^{\mathbb{Z}}$ there exists an orbit $r_\sigma(s):\mathbb{R}\to\mathbb{R}_+$ of \eqref{eq:Hamiltonian} such that , if $\sigma_j=0$
\[
|r_\sigma|_{C^1(I_j)}\geq R
\]
and if $\sigma_j=1$
\[
|r_\sigma-r_{\mathrm{h}}|_{C^1(I_j)} \leq \varepsilon,
\]
Moreover, if $\sigma$ has only a finite number of non zero entries, then $r_\sigma$ is a homoclinic solution.
\end{thm}

Theorem \ref{thm:maintheorem} can be read as follows. For almost all $G\in\mathbb{R}$  there exist an orbit $r_h$ of \eqref{eq:oneparameterHamiltonian} homoclinic to $\gamma_\infty$ such that the following holds. Let $z_*=(r,y,t)=(r_h(0),\dot{r}_h(0),0)\in \mathbb{R}_+\times\mathbb{R}\times\mathbb{T}$, let $z_\infty=(r,y,t)=(\infty,0,0)=\gamma_\infty\cap\{t=0\}\in \mathbb{R}_+\times\mathbb{R}\times\mathbb{T}$ and denote by $\Phi$ the Poincar\'{e} map induced on the section $\{t=0\}$ by the flow to the Hamiltonian \eqref{eq:oneparameterHamiltonian}. Then, for any $\delta>0$ and any sequence $\{z_k\}_{k\in\mathbb{Z}} \subset \{z_\infty,z_*\}^\mathbb{Z}$  there exists a point $z\in B_\delta (z_0)$ and a sequence $\{n_k\}_{k\in\mathbb{Z}}\in\mathbb{N}^\mathbb{Z}$ such that  $\Phi^{n_k}(z_0)\in B_\delta (z_k)$ \footnote{By $B_{\delta}(z_\infty)$ we mean the set $\{|y|\leq \delta,\ |r|^{-1}\leq \delta\}$.}. The statement in Theorem \ref{thm:maintheorem} is indeed stronger since it also provides control on the orbit in all the intervals $[(n_k-n_{k-1})/2,(n_k+n_{k+1})/2]$. 

The following corollary of Theorem \ref{thm:maintheorem} can obtained by nowadays well known arguments (see for example \cite{MR1487629} and  \cite{koz1983}).

\begin{cor}
For almost all $G\in\mathbb{R}$ the Restricted Isosceles $3$-body Problem is not $C^{\omega}$ integrable and  has positive topological entropy.
\end{cor}

\subsection{Outline of the proof: new tools for the study of oscillatory motions}
As in Moser's approach, the first main step in our construction is to prove the existence of a homoclinic orbit to $\gamma_\infty$. Yet, in the setting of Theorem \ref{thm:maintheorem}, geometric perturbation theory is not available since the Hamiltonian system $H_G$ in \eqref{eq:oneparameterHamiltonian} is not nearly integrable. 
Instead, we will adopt a global approach and deploy the powerful machinery of the theory of calculus of variations. In particular, we rephrase the problem of existence of homoclinic orbits to $\gamma_\infty$ as that of the existence of critical points of a certain action functional $\mathcal{A}_G$ (cfr \ref{eq:defactionfunctional}) defined in a suitable Hilbert space $D^{1,2}$ (cfr \eqref{eq:functionalspace}). The existence of critical points of the action functional $\mathcal{A}_G$ is obtained by a minmax argument tailored made for the present problem. The use of minmax techniques to study the existence and multiplicity results for homoclinic orbits in Hamiltonian systems has already been widely exploited in the literature (see for example \cite{Ser92,MR1070929,MR1119200} and \cite{MR1487629}). In the variational approach to our problem, we  face two main difficulties at this step: the phase space is not compact and the vector field presents singularities (corresponding to possible collision with the massive bodies). In order to overcome the first  difficulty we make use of a \emph{renormalized action functional} (see Remark \ref{rem:remarkdefinitionrenormalizedlagrangian}) defined on a appropriately chosen functional space $D^{1,2}$. In order to  avoid singularities and gain compactness we then perform a constrained deformation argument. With these techniques, together with a compactness property of the map $\mathrm{d} \mathcal{A}_G:D^{1,2}\to D^{1,2}$ (Struwe's monotonicity trick), we are able to show that, for almost all values of the angular momentum $G$ \footnote{See the discussion at the beginning of Section \ref{sec:deformation}.},  there exists a Palais-Smale sequence in $D^{1,2}$ which converges to a critical point of the action functional $\mathcal{A}_G$. This proves the existence of an orbit $\tilde{r}_h$ homoclinic to $\gamma_\infty$, which actually correspond to a doubly parabolic motion of our problem. It is worthwhile pointing out that half parabolic and hyperbolic motions for the $n$-body problem have been obtained using variational methods in \cite{MadVen09,MadVen20} with a different technique.

The homoclinc orbit $\tilde{r}_h$ obtained in this way is associated with an intersection between the stable and unstable manifolds of the periodic orbit $\gamma_\infty$.  To proceed further,  though we can not tell whether this intersection is  transversal or not, we may rely on our minmax construction to deduce some topological transversality. This can be achieved  by a  topological degree argument based on a general result by Hofer (\cite{MR843584}). More precisely, we exploit the mountain pass characterization of $\tilde{r}_h$ to  show that for almost all values of the angular momentum $G$ (except possibly a finite set of values) there exists a (possibly different) critical point $r_h$ of the action functional  $\mathcal{A}_G$ for which the Leray-Schauder index of the map $\nabla\mathcal{A}_G:D^{1,2}\to D^{1,2}$ at $r_h$ is well defined and different from zero \footnote{In Proposition \ref{prop:topologicaltransversality} show that the topological degree being non zero implies that the intersection between the invariant manifolds of $\gamma_\infty$ at $r_h$ is topologically transverse.}. This allows us to shadow finite segments of the homoclinic orbit $r_h$. The proof of Theorem \ref{thm:maintheorem} is  then obtained by combining a suitable parabolic version of the Lambda lemma close to $\gamma_\infty$ with the outer dynamics wich shadows  finite segments of $r_h$.

\subsection{Organization of the paper}
In Section \ref{sec:2BodyProblem} we recall some well known facts about the $2$-body Problem. Then, in Section \ref{sec:infinitydynamics} we  analyze the dynamics around the periodic orbit $\gamma_\infty$. In particular, the existence of stable and unstable manifolds $W^\pm(\gamma_\infty;G)$ and a parabolic version of the lambda lemma close to $\gamma_\infty$. In Section \ref{sec:homoclinics} we introduce the variational formulation and prove the existence of a homoclinic orbit to $\gamma_\infty$ by means of a minmax argument. Then, in Section \ref{sec:toptransvers} we obtain a (possibly different) homoclinic orbit associated with a topologically transverse intersection between $W^\pm(\gamma_\infty;G)$. Finally in Section \ref{sec:multibumpsolutions} we combine the parabolic Lambda lemma of Section  \ref{sec:infinitydynamics} together with the robustness of the topological degree under perturbations to construct ``multibump" homoclinics and finish the proof of Theorem \ref{thm:maintheorem}.


\section{The $2$-body Problem}\label{sec:2BodyProblem}
In this section we recall some well known facts about the $2$-body Problem (2BP) which will be used in the following. In polar coordinates, the Hamiltonian of the 2BP reads (compare \eqref{eq:Hamiltonian})
\begin{equation}\label{eq:Hamiltonian2BP}
H_{\mathrm{2BP}}(r,\alpha,y,G)=\frac{y^2}{2}+\frac{G^2}{2r^2}-\frac{1}{r}.
\end{equation}
As for \eqref{eq:Hamiltonian}, the rotational symmetry implies that $G$ is a conserved quantity, so we look at \eqref{eq:Hamiltonian2BP} as a one-parameter family of Hamiltonian functions $H_{\mathrm{2BP},G}(r,y)$. For each $G\in\mathbb{R}$ the Hamiltonian $H_{\mathrm{2BP},G}(r,y)$ is integrable and the motion can be classified in terms of the value of the energy: negative values correspond to elliptic motions, positive energies correspond to hyperbolic motions and for zero energy the motion is parabolic.

It is also straightforward to check that for all $G\in\mathbb{R}$
\[
z_\infty=\{r=\infty,\ y=0\}\subset\mathbb{R}_+\times\mathbb{R}.
\]
is a fixed point for the flow of \eqref{eq:Hamiltonian2BP}\footnote{To analyze this fixed point properly one should work in  McGehee coordinates, which are introduced in Section \ref{sec:parabolicLambdalemma}.}. Moreover, for all $G\in\mathbb{R}$ the fixed point $z_\infty$ posses stable and unstable manifolds which coincide along a one dimensional homoclinic manifold $W^h_{2BP}(z_{\infty},G)$. The homoclinic orbit $W^h(z_{\infty},G)$ is indeed the parabolic orbit of the 2BP with angular momentum $G$.

\begin{lem}\label{lem:homoclinic2BP}
There exist real analytic functions $r_0(u;G)$ and $y_0(u;G)$, defined for all $u\in\mathbb{R}$, such that 
\begin{equation}\label{eq:homoclinicparametrization2BP}
W^h_{\mathrm{2BP}}(z_{\infty};G)=\{ r= r_0(u;G),\ y=y_0(u;G),\ u\in\mathbb{R}\}.
\end{equation}
Moroever, $r_0(u;G)\geq G^2/2$ for all $u\in\mathbb{R}$ and 
\[
r_0(u;G)\sim u^{2/3}\qquad\qquad y_0(u;G)\sim u^{-1/3}\qquad\qquad\text{as}\qquad u\to\pm\infty.
\]
In addition, for any $G,G_*\in\mathbb{R}$ we have
\[
|r_0(u;G)-r_0(u;G_*)|\lesssim |G^2-G_0^2|  \qquad\qquad\text{as}\qquad u\to\pm\infty.
\]
\end{lem}

\begin{rem}\label{rem:remarkdifferentG}
In the last item of Lemma \ref{lem:homoclinic2BP} we compare solutions associated with different values of the angular momentum $G$. The fact that we need that kind of information in our argument is due to a technical step (Struwe's monotonicity trick) in Section \ref{sec:homoclinics} (see Remark \ref{rem:remarkdefinitionrenormalizedlagrangian} and Lemma \ref{lem:monotonicitylemma}).
\end{rem}

\begin{proof}
A proof of the first two items can be found in \cite{MR1284416}, where the authors also show that 
\[
r_h(u;G)=\frac{G^2(\tau^2(u)+1)}{2}\qquad\qquad\text{for}\qquad\qquad u=\frac{G^{3}}{2}\left(\tau(u)+\frac{\tau^3(u)}{3}\right).
\]
One can check that for $\tau\in\mathbb{R}$ the second equality admits the unique inverse
\[
\tau(u)=\left( 3G^{-3}u+\sqrt{9G^{-6}u^2-1}\right)^{1/3}-\left( 3G^{-3}u+\sqrt{9G^{-6}u^2-1}\right)^{-1/3}
\]
which for large $u$ yields that 
\[
\tau(u)= (6G^{-3}u )^{1/3}\left(1+\mathcal{O}(u^{-1})\right).
\]
Therefore, as $u\to\pm\infty$
\[
r_h(u;G)=\frac{G^2}{2}+\frac{(6u)^{2/3}}{2}\left(1+\mathcal{O}(u^{-1})\right)
\]
and the conclusion follows.
\end{proof}

Define the local stable and unstable manifolds\footnote{One can prove that orbits starting at points in $W^{+}_{2BP,loc} (z_\infty;G)$  (respectively $W^{-}_{2BP,loc} (z_\infty;G)$ ) are confined in the region  $\{r> G^2/2,y\geq 0\}$ for all positive times (respectively in the region  $\{r> G^2/2,y\leq 0\}$ for all negative times).}
\[
\begin{split}
W^{+}_{2BP,loc} (z_\infty;G)=&W^h_{2BP}(z_{\infty};G)\cap\{ y>0\}\\
W^{-}_{2BP,loc} (z_\infty;G)=&W^h_{2BP}(z_{\infty};G)\cap\{ y<0\}.
\end{split}
\]
It is a standard fact that $W^{\pm}_{2BP,loc} (z_\infty;G)$ are exact Lagrangian submanifolds so they can therefore be parametrized in terms of a generating function. 

\begin{lem}\label{lem:hamiltonjacobi2bp}
There exists $S_0(r;G):(G^2/2,\infty)\to\mathbb{R}_+$, which satisfies
\[
H_{2BP;G}(r,\partial_r S_0(r;G))=0
\]
and such that 
 \[
W^{\pm}_{2BP,loc} (z_\infty;G)=\{(r,\pm \partial_r S_0(r;G))\in \mathbb{R}_+\times\mathbb{R}\colon r> G^2/2\}.
\]
\end{lem}

\section{The dynamics close to $\gamma_\infty$}\label{sec:infinitydynamics}

In this section we study the dynamics in a neighbourhood of the periodic orbit at infinity defined in \eqref{eq:periodicorbitinfty}. Despite being degenerate (the linearized vector field vanishes at $\gamma_\infty$) the flow close to the periodic orbit $\gamma_\infty$ behaves in a similar way to the flow on a neighbourhood of a hyperbolic periodic orbit.

\subsection{The local invariant manifolds}

Let $\phi_G^s$ be the time $s$ flow associated with the Hamiltonian $H_G$ defined in \eqref{eq:oneparameterHamiltonian}. It is a classical result by McGehee \cite{MR362403} (see also \cite{MR2030148}) that $\gamma_\infty$ posses local stable and unstable invariant manifolds (by $\pi_r,\pi_y$ we denote the projection on the $r$ and $y$ coordinates of a point $(r,y,t)\in\mathbb{R}_+\times\mathbb{R}\times\mathbb{T}$)
\begin{equation}\label{eq:localinvariantmanifolds}
\begin{split}
W^+_{\mathrm{loc},R}(\gamma_\infty;G)=&\{x\in\mathbb{R}_+\times\mathbb{R}\times\mathbb{T}\colon \pi_r \phi^s_G(x)\geq R,\ \pi_y \phi^s(x)\leq 1/R,\ \forall s\geq 0\}\\
W^-_{\mathrm{loc},R}(\gamma_\infty;G)=&\{x\in\mathbb{R}_+\times\mathbb{R}\times\mathbb{T}\colon \pi_r \phi^s_G(x)\geq R,\ \pi_y \phi^s(x)\leq 1/R,\ \forall s\leq 0\}\\
\end{split}
\end{equation}
It is also a standard fact that $W^\pm_{\mathrm{loc},R}(\gamma_\infty;G)$ are exact Lagrangian submanifolds so they can be parametrized in terms of a generating function. The following result follows directly from the arguments in the proof of Theorem 4.4.  in \cite{guardia2021symbolic} (see Remark \ref{rem:remarkexistencegeneratingfunctions}).

\begin{prop}[\cite{guardia2021symbolic}]\label{prop:generatingfunctions}
Let $H_G$ be the one parameter family of Hamiltonians defined in \eqref{eq:oneparameterHamiltonian} and fix any $G_*>0$. Then, there exist $R>0$ such that for all $G\in [-G_*,G_*]$ there exist two functions $S^\pm(r,t;G):[R,\infty)\times\mathbb{T}\to \mathbb{R}$, real analytic on  $r$ and $G$, solutions to the Hamilton-Jacobi equation
\[
H_G(r,t,\partial_r S^\pm(r,t;G))+\partial_t S^\pm(r,t;G)=0
\]
and such that 
\[
W^\pm_{\mathrm{loc},R}(\gamma_\infty;G)=\{(r,y,t)\in\mathbb{R}_+\times\mathbb{R}\times\mathbb{R}\colon r\in [R,\infty),\ y=\partial_r S^\pm (r,t;G) \}.
\]
Moreover, if we let $S_0(r;G)$ be the function defined in Lemma \ref{lem:hamiltonjacobi2bp}, we have that 
\[
S^\pm(r;G)-S_0(r;G)\sim r^{-3/2}\qquad\qquad\text{as} \qquad\qquad r\to\infty.
\]
\end{prop}

\begin{rem}\label{rem:remarkexistencegeneratingfunctions}
In Theorem 4.4. in \cite{guardia2021symbolic} the authors only show the existence of the generating functions $S^\pm(r,t;G)$ for large values of $G$. The reason is that, under the hypothesis of large $G$, they can extend the generating functions to a common domain where they can measure their diference. However, if we are only concerned with the existence and behaviors of the generating functions close to infinity, the problem is already perturbative, and the very same arguments apply to obtain the conlcusion in Proposition \ref{prop:generatingfunctions}.
\end{rem}

Define the global stable and unstable invariant manifolds
\begin{equation}\label{eq:globalmanifolds}
W^+(\gamma_\infty;G)=\bigcup_{s\leq0} \phi^s_G (W^+_{\mathrm{loc},R}(\gamma_\infty;G))\qquad\qquad W^-(\gamma_\infty;G) =\bigcup_{s\geq0} \phi^s_G (W^-_{\mathrm{loc},R}(\gamma_\infty;G)).
\end{equation}
The analytic dependence of the functions $S^\pm(r,t;G)$ on $r$ and $G$ will be key to prove that transversal intersections (whenever they exist)  between the global stable and unstable invariant manifolds \eqref{eq:globalmanifolds} are topologically transverse except for (possibly) a finite subset of values of $G$. This is key for the multibump construction. On the other hand, the estimate  $S^\pm-S^0\sim r^{-3/2}$ as $r\to\infty$ will be needed in the proof of certain technical steps in Lemma \ref{lem:asymptoticactions} (see Appendix \ref{sec:technicallemmas}).

\subsection{The parabolic Lambda Lemma}\label{sec:parabolicLambdalemma}
 We now analyze the topology of the flow lines close to the periodic orbit $\gamma_\infty$. For that, it is convenient to introduce the McGehee transformation $r=2/x^2$ in which the equations of motion associated with the Hamiltonian system $H_G$ in \eqref{eq:oneparameterHamiltonian} read
\[
\dot{x}=-\frac{x^3}{4} \frac{\partial H_G}{\partial y}=-\frac{x^3y}{4}\qquad\qquad\qquad \dot{y}=\frac{x^3}{4} \frac{\partial H_G}{\partial x}=-\frac{x^4}{4} \frac{1}{(1+\frac{x^4\rho^2(t)}{4})^{3/2}}+ \frac{x^6 G^2}{8}.
\]

\begin{figure}\label{fig:McGehee}
\centering
\includegraphics[scale=0.50]{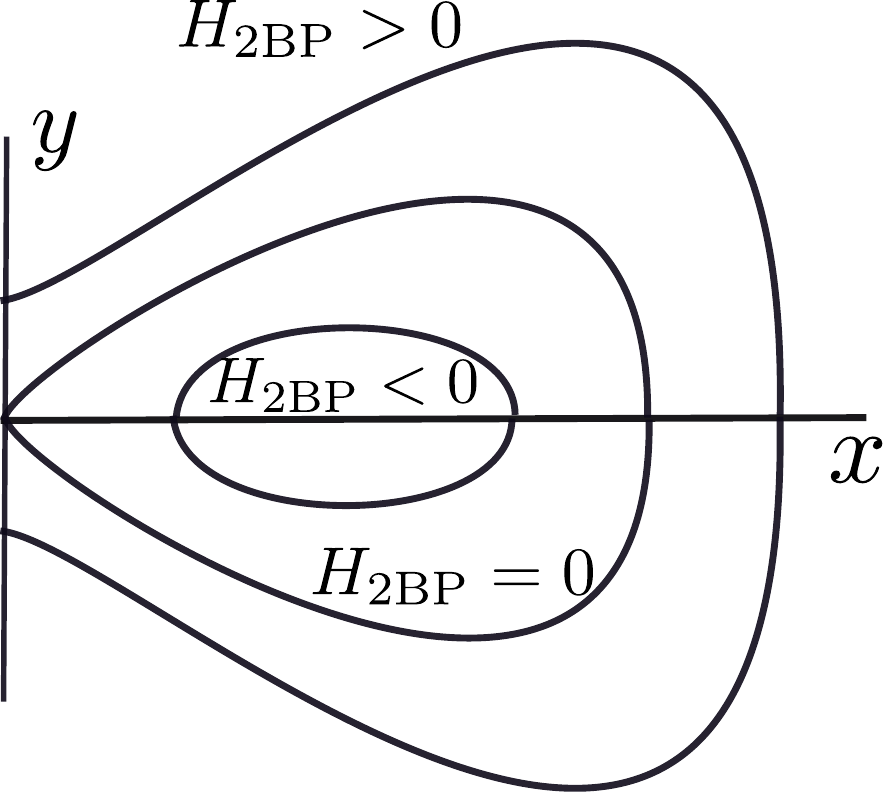}
\caption{Phase portrait of the 2BP in McGehee coordinates. The fixed point $z_\infty$ corresponds in McGehee coordinates to the origin in the $(x,y)\in\mathbb{R}^2$ plane.}
\end{figure}

In this variables, the periodic orbit at infinity \eqref{eq:periodicorbitinfty} now corresponds to the periodic orbit $\hat{\gamma}_\infty=\{x=y=0,\ t\in\mathbb{T}\}$. Following Moser \cite{MR1829194}, we now straighten the stable and unstable directions associated with this periodic orbit. To that end, we introduce the change of variables
\[
\tilde{q}=\frac{x-y}{2}\qquad\qquad \tilde{p}=\frac{x+y}{2}.
\]
In these coordinates
\begin{equation}\label{eq:stableunstabledirections}
\dot{\tilde{q}}=\frac{1}{4}(\tilde{q}+\tilde{p})^3(q+\mathcal{O}_3(\tilde{q},\tilde{p}))\qquad\qquad\qquad \dot{\tilde{p}}=-\frac{1}{4}(\tilde{q}+\tilde{p})^3(\tilde{p}+\mathcal{O}_3(\tilde{q},\tilde{p}))
\end{equation}
so it is clear that the local stable and unstable invariant manifolds associated with the periodic orbit $\tilde{\gamma}_\infty=\{\tilde{q}=\tilde{p}=0,t\in\mathbb{T}\}$, which, by the work of McGehee \cite{MR362403} (see also Proposition \ref{prop:generatingfunctions}) we already know exist, are close (for small $|\tilde{q}|,|\tilde{p}|$)  to $\{\tilde{q}=0\}$ and $\{\tilde{p}=0\}$ respectively. Let now, for sufficiently small $\delta>0$, define the set
\[
Q_\delta=\{(\tilde{q},\tilde{p},t)\in\mathbb{R}^2\times\mathbb{T}\colon |\tilde{q}|\leq \delta,\ |\tilde{p}|\leq \delta\}.
\]
and, let $(0,\tilde{p},t)\in Q_\delta\to (\tilde{p},\gamma^s(\tilde{p},t),t)\subset Q_\delta$ and $(\tilde{q},0,t)\in Q_\delta \to (\tilde{q},\gamma^u(\tilde{q},t),t)\subset Q_\delta$ be graph parametrizations of these local invariant manifolds. Introduce new variables on $Q_\delta$ given by
\[
q=\tilde{q}-\gamma^{s}(\tilde{p},t)\qquad\qquad p=\tilde{p}-\gamma^u(\tilde{q},t).
\]
From the invariance equation satisfied by $\gamma^{u,s}$ one can deduce their Taylor expansion around $\tilde{q}=\tilde{p}=0$. Then, an easy computation, shows that 
\begin{equation}\label{eq:straightenedflow}
\dot{q}=-\frac{q}{4}\left((q+p)^3+\mathcal{O}_4(q,p)\right)\qquad\qquad \dot{p}=\frac{p}{4}\left((q+p)^3+\mathcal{O}_4(q,p)\right)
\end{equation}
so in coordinates $(q,p,t)\subset Q_{\delta}$ the local stable and unstable manifolds are the sets $\{p=0\}\cap Q_{\delta}$ and $\{q=0\}\cap Q_{\delta}$ respectively. Define now, for $a<\delta$ the sections (see Figure \ref{fig:Lambdalemma})
\[
\Sigma^+_{a}=\{(q,p,t)\in Q_{2\delta} \colon p=\delta,\ 0< q\leq a \}\qquad\qquad
\Sigma^-_{a}= \{(q,p,t)\in Q_{2\delta} \times\mathbb{T}\colon q=\delta,\ 0< p\leq a \}
\]
and the associated Poincar\'{e} map $\Phi_{\mathrm{loc}}:\Sigma^+_a\to \Sigma^-_{a'}$, associated with the flow \eqref{eq:stableunstabledirections}, whenever is well defined. Lemma \ref{lem:lambdalemma} shows that a parabolic version of the Lambda Lemma holds for the degenerate periodic orbit $\{p=q=0\}$. In order to build orbits whose final motions are hyperbolic, we also introduce the outer sections 
\[
\Sigma^+_{a,\mathrm{hyp}}=\{(q,p,t)\in Q_{2\delta} \colon p=\delta,\ -a\leq q<0 \}\qquad\qquad
\Sigma^-_{a,\mathrm{hyp}}= \{(q,p,t)\in Q_{2\delta} \colon q=\delta,\ -a\leq p<0 \}.
\]
\begin{figure}\label{fig:Lambdalemma}
\centering
\includegraphics[scale=0.60]{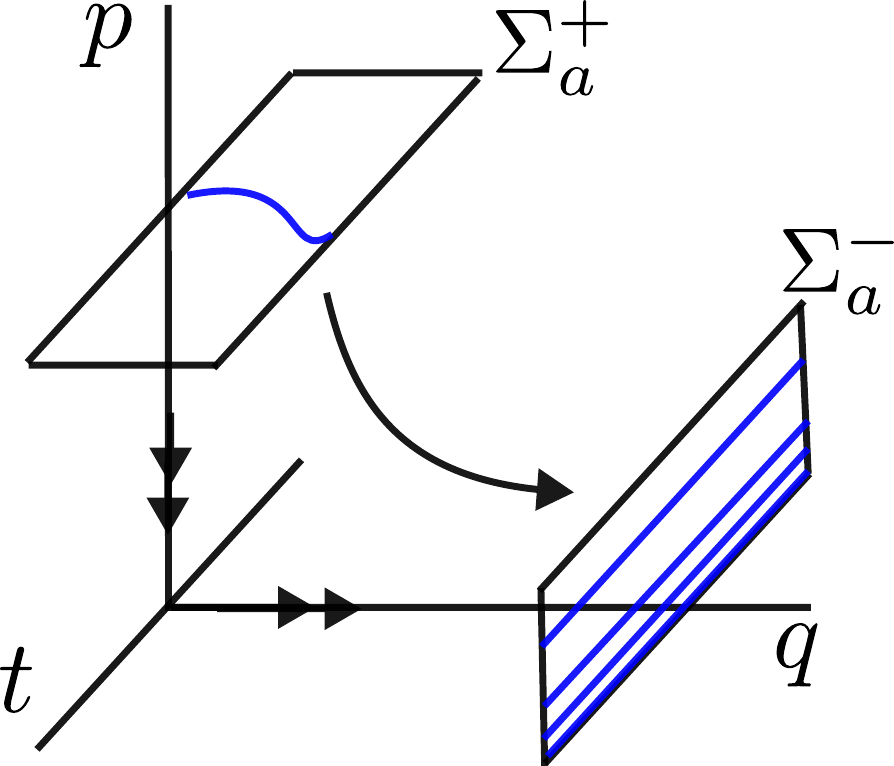}
\caption{The sections $\Sigma_a^\pm$. The Poincar\'{e} map $\Phi_{\mathrm{loc}}:\Sigma^+_a\to \Sigma^-_{a'}$ sends the blue line in the section  $\Sigma_a^+$ into the blue line in the section $\Sigma_a^-$, which accumulates to $\{p=0\}$.}
\end{figure}

\noindent The proof of the following proposition follows plainly from the arguments in Chapter IV of \cite{MR1829194}, where an analogous result is proved for the Sitnikov problem. See also Theorem 5.4. in \cite{https://doi.org/10.48550/arxiv.2207.14351}.
\begin{lem}\label{lem:lambdalemma}
Fix any $G_*>0$. Then, there exists $C>0$ sufficiently large and $\delta>0$ sufficiently small such that for any $G\in [-G_*,G_*]$ and any $a\in (0,\delta/2)$ the Poincar\'{e} map 
\[
\Phi_{\mathrm{loc}}: \Sigma^+_a\longrightarrow \Sigma^-_{a^{1-C\delta}}
\]
is well defined. Moreover, for any $t_1$ sufficiently large there exist unique $q$ and $p_1$, which satisfy
\[
q^{1+C\delta}\leq p_1\leq q^{1-C\delta}\qquad\qquad\quad q^{-3(1-C\delta)/2}\lesssim t_1 \lesssim q^{-3(1+C\delta)/2},
\]
for which $\Phi_{\mathrm{loc}}(q,a,0)=(a,p_1,t_1)$.

In addition, for any $(q,a,0)\in\Sigma^+_{a,\mathrm{hyp}}$ (respectively $(a,p,0)\in\Sigma^-_{a,\mathrm{hyp}}$), the orbit $( q_{\mathrm{hyp}}(s),p_{\mathrm{hyp}}(s),s)$ of \eqref{eq:straightenedflow} with initial condition $(q,a,0)$  (respectively $(a,p,0)$) is defined for all forward (respectively backward) times and satisfies 
\[
\lim_{s\to\infty} y(q_{\mathrm{hyp}}(s),(p_{\mathrm{hyp}}(s))>0\qquad \qquad (\text{respectively}\ \lim_{s\to-\infty}y(q_{\mathrm{hyp}}(s),(p_{\mathrm{hyp}}(s))<0).
\]
\end{lem}

The first item in Lemma \ref{lem:lambdalemma} shows  that the iterates of curves which are transversal to the local stable manifold accumulate along the unstable manifold  (see also Figure \ref{fig:Lambdalemma}). The second item ensures that orbits with initial conditions on $\Sigma^+_{a,\mathrm{hyp}}$ (respectively $\Sigma^-_{a,\mathrm{hyp}}$) have forward (respectively backward) hyperbolic final motions. We now translate these results to the original coordinates. To that end we introduce the sections
\[
\begin{split}
\Lambda^+_{R,\delta}=&\{(r,y,t)\colon r=R,\ 0<\partial_r S^+(R,t;G)-y\leq \delta,\  t\in\mathbb{T}\}\\ \Lambda^-_{R,\delta}=&\{(r,y,t)\colon r=R,\ 0< y-\partial_r S^-(R,t;G)\leq \delta,\  t\in\mathbb{T}\}
\end{split}
\]
and the map $\Phi_{\mathrm{loc},R_1,R_2}:\Lambda^+_{R_1,\delta}\to \Lambda^-_{R_2,\delta'}$ whenever is well defined. We also define the sections leading to hyperbolic final motions\[
\begin{split}
\Lambda^+_{R,\delta}=&\{(r,y,t)\colon r=R,\ -\delta\leq \partial_r S^+(R,t;G)-y< 0,\  t\in\mathbb{T}\}\\ \Lambda^-_{R,\delta}=&\{(r,y,t)\colon r=R,\ -\delta \leq y-\partial_r S^-(R,t;G)<0,\  t\in\mathbb{T}\}.
\end{split}
\]

\begin{lem}\label{lem:lambdalemmaoriginalcoord}
Fix any $G_*>0$. Then, there exist $R>0$ sufficiently large  such that for any $R_1,R_2\geq R$ there exists $\delta_0(R_1,R_2)$ such that for all $G\in [-G_*,G_*]$ the Poincar\'{e} map 
\[
\Phi_{\mathrm{loc},R_1,R_2}:\Lambda^+_{R_1,\delta}\to \Lambda^-_{R_2,\delta'}
\]
is well defined for $\delta\leq \delta_0$ and some $\delta'(R_1,R_2,\delta)>0$. There exists $T_*$ such that for any $T\geq T_*$ there exist unique $y_0,y_1$ such that $\Phi_{\mathrm{loc},R_1,R_2}(R_1,y_0,0)=(R_2,y_1,T)$. Moreover, for any $\varepsilon>0$ there exists $T_{**}$ such that, if $T\geq T_{**}$ and $\Phi_{\mathrm{loc},R_1,R_2}(R_1,y_0,0)=(R_2,y_1,T)$, then 
\[
\partial_r S^+(R_1,0;G)-y_0\leq \varepsilon \qquad\qquad y_1-\partial_r S^+(R_2,T;G)\leq \varepsilon.
\]
In addition, the orbit $( r_{\mathrm{hyp}}(s),y_{\mathrm{hyp}}(s),s)$ of \eqref{eq:oneparameterHamiltonian} with initial condition  $(R_1,y,0)\in\Lambda^+_{R_1,\delta,\mathrm{hyp}}$ (respectively $(R_2,y,0)\in\Lambda^-_{R_2,\delta,\mathrm{hyp}}$), is defined for all forward (respectively backward) times and satisfies 
\[
\lim_{s\to\infty} y_{\mathrm{hyp}}(s)>0\qquad \qquad (\text{respectively}\ \lim_{s\to-\infty}y_{\mathrm{hyp}}(s)<0).
\]
\end{lem}


\section{Existence of homoclinic orbits to $\gamma_\infty$}\label{sec:homoclinics}

In this section we establish the existence of orbits of the Hamiltonian \eqref{eq:Hamiltonian}, which are homoclinic to $\gamma_\infty$. For $|G|\gg1$, the Hamiltonian \eqref{eq:Hamiltonian} can be considered as a perturbation of the integrable 2BP, in which there exists a homoclinic manifold to $\gamma_\infty$ (see Lemma \ref{lem:homoclinic2BP}). Therefore,  for $|G|\gg 1$, one can use geometric perturbation theory to prove that the global invariant manifolds $W^+(\gamma_\infty;G)$ and $W^-(\gamma_\infty;G)$ defined in \eqref{eq:globalmanifolds} intersect transverally. This was the approach used in \cite{guardia2021symbolic} where the following result was proved.

\begin{thm}[\cite{guardia2021symbolic}]\label{thm:transversalityjde}
There exists $G_*<\infty$ such that for all $G$ such that $|G|\geq G_*$ the global stable and unstable manifolds  $W^+(\gamma_\infty;G)$ and $W^-(\gamma_\infty;G)$ defined in \eqref{eq:globalmanifolds}, intersect transversally.
\end{thm}

Yet, for a fixed $G\in\mathbb{R}$, the Hamiltonian  \eqref{eq:Hamiltonian} is not close to the 2BP. Therefore, geometric perturbation theory cannot help to study the existence of transversal  intersections between $W^+(\gamma_\infty;G)$ and $W^-(\gamma_\infty;G)$. We however exploit the variational formulation of the problem, in which the powerful techniques from nonlinear functional analysis are available.

More concretely, in Section \ref{sec:variationalformulation} we introduce a suitable action functional, defined on a suitable Hilbert space, whose critical points are indeed orbits of \eqref{eq:Hamiltonian} which are homoclinic to $\gamma_\infty$. Then, in Section \ref{sec:criticalpointsaction} we establish the existence of a critical point of the aforementioned action functional using a minmax argument. The minmax characterization of the critical point obtained is crucial for the construction in Section \ref{sec:multibumpsolutions}.

\subsection{The Variational Formulation}\label{sec:variationalformulation}

We introduce the vector space of real valued functions
\begin{equation}\label{eq:functionalspace}
D^{1,2}=\{\varphi\in C(\mathbb{R})\colon \exists v_\varphi \in L^2(\mathbb{R})\ \text{such that }\varphi(s)=\varphi(0)+\int_0^s v_\varphi(t)\mathrm{d}t\ \  \forall s\in\mathbb{R}\}.
\end{equation}
In the following, we will write $\dot{\varphi}=v_\varphi$ (i.e. $v_\varphi$ is the weak derivative of $\varphi$). It is easy to chek that 
\[
\langle \varphi,\psi \rangle_{D^{1,2}}= |\varphi(0)\psi(0) |+ \langle \dot{\varphi}, \dot{\psi} \rangle_{L^2}
\]
defines an inner product on $D^{1,2}$ for which the functional space $D^{1,2}$ equiped with this inner product is a Hilbert space. We write 
\[
\lVert \varphi \rVert_{D^{1,2}}= \left( \langle \varphi,\varphi \rangle_{D^{1,2}} \right) ^{1/2}.
\]
for the induced norm. Notice that for all $\varphi\in D^{1,2}$ and all $s\in\mathbb{R}$ 
\[
|\varphi(s)|\leq |\varphi(0)|+\lVert \dot{\varphi} \rVert_{L^2} \sqrt{|s|}.
\]
After the introduction of the functional space $D^{1,2}$ it is an easy computation to show that the existence of orbits of \eqref{eq:oneparameterHamiltonian} homoclinic to the periodic orbit at infinity $\gamma_\infty=\{r=\infty,y=0,t\in\mathbb{T}\}$  is equivalent to the existence of critical points of the action functional $\mathcal{A}_G:D^{1/2}\to \mathbb{R}$ given by 
\begin{equation}\label{eq:defactionfunctional}
\mathcal{A}_{G}(\varphi;G_0)=\int_{\mathbb{R}} \mathcal{L}_{\mathrm{ren}}(\varphi,\dot{\varphi},s;G,G_0)\mathrm{d}s,
\end{equation}
where
\begin{equation}\label{eq:renormalizedLagrangian}
\mathcal{L}_{\mathrm{ren}}(\varphi,\dot{\varphi},s;G,G_0)=\frac{\dot{\varphi}^2}{2}+V_G(r_0+\varphi)-V_0(r_0)-\ddot{r}_0\varphi,
\end{equation}
$V_G$ stands for the effective potential 
\begin{equation}\label{eq:effectivepotential}
V_{G}(r,t)=\frac{G^2}{2r^2}-\frac{1}{\sqrt{r^2+\rho^2(t)}},
\end{equation}
and $V_0(r_0)=\frac{G_0^2}{2r_0^2}-\frac{1}{r_0}$ with $r_0$ being the parabolic orbit of the 2BP with angular momentum $G_0\in\mathbb{R}$ (see Remark \ref{rem:remarkG0}).

\begin{rem}\label{rem:remarkdefinitionrenormalizedlagrangian}
$\mathcal{L}_{\mathrm{ren}}(\varphi,\dot{\varphi},s;G)$ is indeed a \textit{renormalized} Lagrangian, that is, we have substracted the term $V_0(r_0)$ in the integrand of what would be the ``natural" action functional. The reason behind the definition of \eqref{eq:defactionfunctional} is that the action of a parabolic orbit is infinite. Indeed, the Lagrangian of the 2BP reads
\[
\mathcal{L}_0(r_0,\dot{r}_0)=\frac{\dot{r}_0^2}{2}-V_0(r_0)
\]
and for a parabolic orbit $r_0(s)\sim s^{2/3}$ for $s\to\pm\infty$.
\end{rem}

\begin{rem}\label{rem:remarkG0}
It might seem surprising that when defining the renormalized Lagrangian $\mathcal{L}_{\mathrm{ren}}$, we  let $G_0$ be an independent parameter instead of taking $G=G_0$. The reason is that in this way, for a fixed $G_0\in \mathbb{R}$ and fixed $\varphi\in D^{1,2}$ the function $G\to \mathcal{A}_G(\varphi)$ is monotonely decreasing. This will allow us to use a monotonicity trick due to Struwe which is key to obtain uniform bounds for certain (Palais-Smale) sequences $\{\varphi_n\}_{n\in\mathbb{N}} \subset D^{1,2}$ for which $\mathrm{d} \mathcal{A}_G(\varphi_n)\to 0$ (see Section \ref{sec:deformation} and, in particular, \ref{lem:monotonicitylemma}). On the other hand, the asymptotic behavior of parabolic solutions as $s\to\pm\infty$ becomes independent of the value of the angular momentum $G$ (see Lemma \ref{lem:homoclinic2BP}) so the definition of the renormalized Lagrangian $\mathcal{L}_{\mathrm{ren}}$ makes sense for $G\neq G_0$. 
\end{rem}

\begin{rem}
Throughout the rest of the paper the value $G_0\in \mathbb{R}_+$ will be fixed. Thus, we omit the dependence of all quantities on $G_0$. Having fixed $G_0\in \mathbb{R}_+$,   we state results for $G\in[-G_0,G_0]$ (or full measure subsets of this set). This choice is completely arbitrary: the results proved below are certainly true if we replace $[-G_0,G_0]$ by any other bounded subset. However, since we have always the freedom to choose $G_0$ as large as we want it is enough to state results for $G\in[-G_0,G_0]$.
\end{rem}\vspace{0.35cm}

The following observation will play an important role in our construction.

\begin{lem}\label{lem:periodicitylemma}
Let $\tau\in\mathbb{Z}$ and define the translation operator 
\[
T_\tau(\varphi)(s)=\varphi(s+\tau)+r_0(s+\tau)-r_0(s).
\]
Then, for all $\tau\in\mathbb{Z}$
\[
\mathcal{A}_G(T_{\tau}(\varphi))=\mathcal{A}_G(\varphi).
\]
\end{lem}

We now state a  technical lemma which will prove useful in later compactness arguments.

\begin{lem}\label{lem:embeddinglemma}
Let $\gamma\geq0$ and let $L^2_\gamma$ be the weighted $L^2$ space with norm given by
\[
\lVert \varphi \rVert_{L^2_\gamma}=\left( \int_\mathbb{R} \frac{|\varphi|^2}{r_0^{3+\gamma}}\right)^{1/2}.
\]
Then, $D^{1,2}$ is continuously embedded in $L^2_\gamma$ for $\gamma\geq0$ and compactly embedded in $L^{2}_\gamma$ for $\gamma>0$.
\end{lem}
\begin{proof}
The proof of the continuous embedding for $\gamma\geq 0$ is obtained by the very same argument used in the proof of Proposition 3.2. in \cite{MR4307034} taking into account that $r_0(s)\sim s^{2/3}$ for $s\to\pm \infty$ and $r_0(s)\geq G_0^2/2\ \forall s\in\mathbb{R}$. We now prove that the embedding for $\gamma>0$ is compact. Take any bounded sequence $\{\varphi_n\}_{n\in\mathbb{N}}\subset D^{1,2}$ such that $\varphi_n\to 0$ weakly in $D^{1,2}$. In particular $\varphi_n(s)\to0$ pointwise for all $s\in\mathbb{R}$.  Since, for any $\varphi\in D^{1,2}$ and any $s\in\mathbb{R}$ we have 
\[
|\varphi(s)|\leq |\varphi(0)|+\lVert \dot{\varphi} \rVert_{L^2}\sqrt{|s|}
\]
we obtain that  for all $s\in\mathbb{R}$
\[
\frac{|\varphi_n(s)|^2}{r_0^{3+\gamma}(s)}\lesssim\frac{\lVert \varphi_n \rVert^2_{D^{1,2}}}{1+ |s|^{1+\gamma}}.
\]
Therefore, a direct application of the dominated convergence theorem shows that 
\[
\lim_{n\to\infty}\lVert \varphi_n\rVert_{L_\gamma^2}^2=\lim_{n\to\infty}\int_{\mathbb{R}} \frac{|\varphi_n|^2}{r_0^{3+\gamma}(s)}=0.
\]
\end{proof}

We now show that $\mathcal{A}_G$ is continuous and has a continuous differential on a suitable subset $Q\subset D^{1,2}$.
\begin{lem}\label{lem:C1lemma}
Let $K>0$ and $\underline{m}>0$ be two fixed constants and define
\[
Q=\{ \varphi\in D^{1,2}\colon \lVert \varphi \rVert_{D^{1,2}}\leq K,\ \min_{s\in\mathbb{R}} r_0(s)+\varphi(s)\geq \underline{m} \}
\]
Then, for any $G\in [-G_0,G_0]$ we have $\mathcal{A}_{G}\in C^1( \mathrm{int}(Q) ,\mathbb{R})$.
\end{lem}

\begin{proof}
Let $\varphi,\psi \in Q$ and make use of the mean value theorem to write
\begin{equation}\label{eq:differenceactions}
\mathcal{A}_G(\varphi)-\mathcal{A}_G(\psi)=\int_{\mathbb{R}} \frac{1}{2}(\dot{\varphi}+\dot{\psi})(\dot{\varphi}-\dot{\psi})+ \partial_r V_G ( r_0+\xi) (\varphi-\psi)-\ddot{r}_0 (\varphi-\psi)
\end{equation}
where $\xi=\lambda\varphi+(1-\lambda)\psi$ for some $\lambda(s)\in[0,1]$. On one hand, 
\[
\left|\int_{\mathbb{R}} (\dot{\varphi}+\dot{\psi})(\dot{\varphi}-\dot{\psi})\right| \leq \left( \int_{\mathbb{R}}|\dot{\varphi}+\dot{\psi}|^2 \right)^{1/2} \left(\int_{\mathbb{R}} |\dot{\varphi}-\dot{\psi}|^2 \right)^{1/2}\to 0
\]
as $\lVert \varphi-\psi \rVert_{D^{1,2}}\to 0$. On the other hand, since for $\varphi,\psi\in D^{1,2}$
\[
\begin{split}
\min_{s\in\mathbb{R}}\  r_0(s)+\xi(s)=&\min_{s\in\mathbb{R}}\ r_0(s)+\lambda \varphi(s)+(1-\lambda)\psi(s)=\min_{s\in\mathbb{R}}\ \lambda(r_0(s)+\varphi(s))+(1-\lambda)(r_0(s)+\psi(s))\\
\geq & \min_{s\in\mathbb{R}}\ \lambda \underline{m}+(1-\lambda) \underline{m}=\underline{m}>0
\end{split}
\]
and convergence in $D^{1,2}$ implies uniform convergence in compact intervals, we have, taking into account the expression of $V_G$ \eqref{eq:effectivepotential}, that
\[
\left( \partial_r V_G ( r_0+\xi) -\ddot{r}_0 \right) (\varphi-\psi)\to 0
\]
pointwise as $\lVert \varphi-\psi \rVert_{D^{1,2}}\to 0$. Moreover, for $s\to\pm\infty$
\[
r_0(s)+\varphi(s)\geq r_0(s)-\left(|\varphi(0)|+\lVert\dot{\varphi} \rVert_{L^2} \sqrt{|s|}\right)\sim s^{2/3}
\]
so, from the definition of $V_G$ in \eqref{eq:effectivepotential}, a straightforward computation shows that for $r_0\to\infty$
\[
\partial_r V_G ( r_0) -\ddot{r}_0\sim r_0^{-3}.
\]
Thus, using again that $\min_{s\in\mathbb{R}}\  r_0(s)+\xi(s)\geq \underline{m}>0$, we obtain the existence of $C>0$ depending only on $K$ and $\underline{m}$ such that for all $s\in\mathbb{R}$
\[
|\partial_r V_G ( r_0(s)+\xi(s)) -\ddot{r}_0(s)|\leq C r_0^{-3}(s)
\]
Therefore,
\[
\begin{split}
\left|\int_{\mathbb{R}}\left( \partial_r V_G ( r_0+\xi) -\ddot{r}_0 \right) (\varphi-\psi)\right| \leq & \left( \int_{\mathbb{R}}|\left( \partial_r V_G ( r_0+\xi) -\ddot{r}_0 \right)| \right)^{1/2} \\
&\times \left(\int_{\mathbb{R}} |\left( \partial_r V_G ( r_0+\xi) -\ddot{r}_0 \right)| |\varphi-\psi|^2 \right)^{1/2}\\
\leq & C \int_\mathbb{R}\frac{|\varphi-\psi|^2}{r_0^3}=C\lVert \varphi-\psi \rVert_{L^2_0},
\end{split}
\]
and the continuity of the map $\mathcal{A}_G:Q\subset D^{1,2}\to \mathbb{R}$ is implied by Lemma \ref{lem:embeddinglemma}. The proof that $\mathrm{d}\mathcal{A}_G:Q\subset D^{1,2}\to D^{1,2}$ is a continuous map follows from similar arguments.
\end{proof}

\begin{lem}\label{lem:propermap}
Let $K>0$ and $\underline{m}>0$ be two fixed constants and let $Q\subset D^{1,2}$ be the subset defined in Lemma \ref{lem:C1lemma}. Then, for any  for any $G\in [-G_0,G_0]$,  $\mathrm{d}\mathcal{A}_G:\mathrm{int}(Q)\to D^{1,2}$ is a compact perturbation of the identity. In partiuclar, this implies that for any compact set $F\subset D^{1,2}$ the set $Q\cap (\mathrm{d}\mathcal{A}_{G})^{-1}(F)$ is compact.
\end{lem}

\begin{proof}
We write
\begin{equation}\label{eq:identitypluscompact}
\begin{split}
\mathrm{d} \mathcal{A}_{G} (\varphi)[\psi]= &\langle \dot{\varphi},\dot {\psi} \rangle_{L^2} -\int_{\mathbb{R}}\left( \frac{r_0+\varphi}{((r_0+\varphi)^2+\rho^2)^{3/2}}-\frac{1}{r_0^2} \right)\psi +  \int_{\mathbb{R}}\left(\frac{G^2}{(r_0+\varphi)^3}-\frac{G_0}{r_0^3}\right) \psi \\
=&\langle \dot{\varphi},\dot {\psi} \rangle_{L^2} + 2 \int_{\mathbb{R}} \frac{\varphi\psi}{r_0^3} -\int_{\mathbb{R}}\left( \frac{r_0+\varphi}{((r_0+\varphi)^2+\rho^2)^{3/2}}-\frac{1}{r_0^2} +\frac{2\varphi}{r_0^3}\right)\psi\\
&+\int_{\mathbb{R}}\left( \frac{G^2}{(r_0+\varphi)^3}-\frac{G_0}{r_0^3} \right)\psi \\
=& \langle \dot{\varphi},\dot {\psi} \rangle_{L^2} + 2 \int_{\mathbb{R}} \frac{\varphi\psi}{r_0^3}+ P(\varphi)[\psi]
\end{split}
\end{equation}
where we have introduced the functional
\[
P(\varphi)[\psi] = \int_{\mathbb{R}}\left( \frac{G^2}{(r_0+\varphi)^3}-\frac{G_0}{r_0^3} \right)\psi -\int_{\mathbb{R}}\left( \frac{r_0+\varphi}{((r_0+\varphi)^2+\rho^2)^{3/2}}-\frac{1}{r_0^2} +\frac{2\varphi}{r_0^3}\right)\psi
\]
Thanks to Lemma \ref{lem:embeddinglemma} we can take 
\[
\langle\langle \varphi,\psi \rangle\rangle_{D^{1,2}}= 2 \int_{\mathbb{R}} \frac{\varphi\psi}{r_0^3}+ \langle \dot{\varphi}, \dot{\psi} \rangle_{L^2}
\]
as an equivalent inner product in $D^{1,2}$. It follows from Lemma \ref{lem:C1lemma} that for all $\varphi\in Q$,  $\mathrm{d}\mathcal{A}_{G}(\varphi): D^{1,2}\to\mathbb{R}$ and $P(\varphi): D^{1,2}\to\mathbb{R}$ are continuous linear functionals and thanks to Riesz representation theorem, for every $\varphi\in Q\subset D^{1,2}$  there exist unique $\eta_A(\varphi), \eta_P (\varphi)\in D^{1,2}$ such that 
\[
 \langle\langle \eta_A(\varphi),\psi \rangle\rangle_{D^{1,2}}= \mathrm{d}\mathcal{A}_{G}(\varphi)[\psi]\qquad\qquad \langle\langle \eta_P(\varphi),\psi \rangle\rangle_{D^{1,2}}= P(\varphi)[\psi].
\]
and $\eta_A=\mathrm{Id}+\eta_P$. After writing
\[
\langle\langle \eta_P(\varphi_*)-\eta_P(\varphi), \eta_P(\varphi_*)-\eta_P(\varphi) \rangle\rangle_{D^{1,2}}= P(\varphi_*)[\eta_P(\varphi_*)-\eta_P(\varphi)]- P(\varphi)[\eta_P(\varphi_*)-\eta_P\varphi)],
\]
a tedious but straightforward computation shows that  for any $\varphi_*,\varphi\in Q$ 
\begin{equation}
\lVert \eta_P(\varphi_*)-\eta_P(\varphi) \rVert_{D^{1,2}}^2 \leq \lVert \varphi_*-\varphi \rVert_{L^2_{1/4}}\lVert \eta_P (\varphi^*)-\eta_P(\varphi) \rVert_{L^2_{1/4}}\leq \lVert \varphi_*-\varphi \rVert_{L^2_{1/4}}\lVert \eta_P (\varphi^*)-\eta_P(\varphi) \rVert_{D^{1,2}}
\end{equation}
what implies that $\eta_P:Q\to D^{1,2}$ is a compact operator (recall that the embedding of $D^{1,2}$ in $L^2_{1/4}$ is compact). The second item in the lemma  plainly follows after writing
\[
\eta_A(\varphi)=\varphi+ \eta_P (\varphi).
\]
Indeed, for a sequence $\{\varphi_n\}_{n\in\mathbb{N}}\subset Q\subset D^{1,2}$ whose image under $\mathrm{d} \mathcal{A}_{G}$ is contained in a compact subset $F\subset D^{1,2}$  there exists a subsequence (which we do not relabel) for which $\{\eta_A(\varphi_n)\}_{n\in\mathbb{N}}$ is convergent in $D^{1,2}$. Then, the proof is finished since $\eta_P$ being a compact operator and implies that (up to passing to a further subsequence) $\{\eta_P(\varphi_n)\}_{n\in\mathbb{N}}$ is also convergent in $D^{1,2}$.
\end{proof}

\noindent From now on we will omit the subscript in the inner product and norm defined in $D^{1,2}$.

\subsection{Existence of critical points of the action functional}\label{sec:criticalpointsaction}
In this section we prove the existence of critical points of the action functional $\mathcal{A}_G$ defined in  \eqref{eq:defactionfunctional} using a minmax argument. In particular, we will employ a constrained version of the celebrated mountain pass theorem of Ambrosetti and Rabinowitz \cite{MR0370183}. The first step is to verify that the level sets of $\mathcal{A}_G$ have  a mountain pass geometry. This is the content of the following proposition.

\begin{prop}
Take any constant $M>0$. Then, for all $G\in[-G_0,G_0]\setminus \{0\}$ there exist $\psi_1,\psi_2\in D^{1,2}$ such that 
\[
\mathcal{A}_{G}(\psi_i)\leq -M\qquad\qquad i=1,2.
\]
Moreover, there exists $M^*>0$ such that if we take $M\geq M^*$, then for any curve $\gamma\in C([0,1],D^{1,2})$ joining $\psi_1$ and $\psi_2$ there exist a point $\psi_\gamma$ for which 
\[
\mathcal{A}_{G}(\psi_\gamma)\geq -M/2.
\]
\end{prop}

\begin{proof}
Let $\mu>0$ so
\[
\begin{split}
\mathcal{A}_G(\mu)=&\int_{\mathbb{R}} V_G(r_0+\mu)-V_0(r_0)=\int_{\mathbb{R}} \frac{1}{((r_0+\mu)^2+\rho^2)^{1/2}}-\frac{1}{r_0}-\frac{G^2}{2(r_0+\mu)^2}+ \frac{G_0^2}{2r_0^2}\\
\leq & \int_{\mu} \frac{1}{r_0+\mu}-\frac{1}{r_0}.
\end{split}
\]
It follows from Fatou's lemma that 
\[
\limsup_{\mu\to\infty} \mathcal{A}_G(\mu)=\limsup_{\mu\to\infty} \int_{\mathbb{R}}  \frac{1}{r_0+\mu}-\frac{1}{r_0}\leq -\int_{\mathbb{R}} \frac{1}{r_0}=-\infty.
\]
On the other hand, take $\eta\in (0,1/2)$. Then, for some finite (and uniform for $\eta\in (0,1/2)$)  $C>0$ we have
\[
\begin{split}
\mathcal{A}_G(\eta)=&\int_{\mathbb{R}} V_G(r_0+\eta)-V_0(r_0)= \int_{\mathbb{R}}  \frac{1}{((r_0+\eta)^2+\rho^2)^{1/2}}-\frac{1}{r_0}+\frac{G_0^2}{2r_0^2}-\frac{G^2}{2(r_0+\eta)^2}\\
\leq& C+\int_{0}^1\frac{1}{r_0+\eta}-\frac{G^2}{2(r_0+\eta)^2}.\\
\end{split}
\]
Using that $r_0(s)=1/2+s^2+\mathcal{O}(s^3)$  for $s\to0$ (this can be deduced from the proof of Lemma \ref{lem:homoclinic2BP}) one can easily check that 
\[
\limsup_{\eta\to 1/2} \mathcal{A}_G(\eta)=-\infty.
\]
The first part of the lemma is proven by taking $\psi_1=\mu$ with $\mu$ large enough and $\psi_2=\eta$ with $\eta\to 1/2$. In order to prove the second item of the lemma we let $R>0$ be such that 
\[
\partial^2_{rr} V_G(r)\geq 0\qquad\qquad \forall r\geq R
\]
and denote by $T$ the value of $s$ for which $r_0(s)\geq R$ for all $s$ such that $|s|\geq T$. Notice that $R$ exists because of the convexity of $V_G(r)$ for large values of $r$, which can be checked explicitely from the expression of $V_G$ in \eqref{eq:effectivepotential}. We now take $\varphi\in D^{1,2}$ such that $\min_{s\in\mathbb{R}} r_0(s)+\varphi(s)=R$. We claim that $\mathcal{A}_G(\varphi)\geq -M/2$ so the lemma follows since, by continuity, for all $\gamma\in C([0,1],D^{1,2})$ joining $\psi_1$ and $\psi_2$ there exist a point $\varphi\in \gamma$ for which 
\[
\min_{s\in\mathbb{R}} r_0(s)+\varphi(s)=R.
\]
We now prove the claim. Lemma \ref{lem:periodicitylemma} implies that, withouth lost of generality, we can suppose that the minimum is attained at the interval $s\in [0,1]$. We express
\[
\mathcal{A}_G(\varphi)=\frac{\lVert \dot{\varphi}\rVert_{L^2}^2}{2}+  J_\leq (\varphi)+J_\geq (\varphi)+E(\varphi)
\]
where
\[
\begin{split}
J_\geq (\varphi)=&\int_{|s|\geq T}\frac{1}{((r_0+\varphi)^2+\rho^2)^{1/2}}-\frac{1}{(r_0^2+\rho^2)^{1/2}}+\frac{r_0\varphi}{(r_0^2+\rho^2)^{3/2}}\\
&-G^2\int_{|s|\leq T} \frac{1}{2(r_0+\varphi)^2}-\frac{1}{2r_0^2}+\frac{\varphi}{r_0^3}\\
J_\leq(\varphi)=&\int_{|s|\leq T}\frac{1}{((r_0+\varphi)^2+\rho^2)^{1/2}}-\frac{1}{(r_0^2+\rho^2)^{1/2}}+\frac{r_0\varphi}{(r_0^2+\rho^2)^{3/2}}\\
&-G^2\int_{|s|\geq T} \frac{1}{2(r_0+\varphi)^2}-\frac{1}{2r_0^2}+\frac{\varphi}{r_0^3}
\end{split}
\]
and
\[
\begin{split}
E(\varphi)=&\int_{s\in\mathbb{R}}\frac{1}{(r_0^2+\rho^2)^{1/2}}-\frac{r_0\varphi}{(r_0^2+\rho^2)^{3/2}}-\frac{1}{r_0}+\frac{\varphi}{r_0^2}+(G_0^2-G^2)\left(\frac{1}{2r_0^2}-\frac{\varphi}{r_0^3}\right)
\end{split}
\]
For the first term, after applying the mean value theorem twice, we obtain that 
\[
J_\geq(\varphi) = \int_{|s|\geq T} \partial^2_{rr} V_G(r_0+\xi) \eta\varphi
\]
with $\eta=\sigma \varphi$, $0\leq\sigma\leq 1$ and $\xi=\lambda \eta$, $0\leq \lambda\leq 1$. Since
\[
\min (r_0+\xi)\geq \min (r_0,r_0+\varphi)\geq R
\]
we have $J_\geq(\varphi)\geq0$ by the definition of $R$. For the second term we use that $\min_{s\in\mathbb{R}} r_0(s)+\varphi(s)\geq R>0$ and that for all $s\in\mathbb{R}$ we have $|\varphi(s)|\leq |\varphi(0)|+ \lVert \dot{\varphi}\rVert_{L^2} \sqrt{|s|}$ so we obtain
\[
\begin{split}
J_\leq (\varphi)\geq &-C+\int_{|s|\leq T} \frac{r_0\varphi}{(r_0^2+\rho^2)^{3/2}}- G^2 \int_{|s|\leq T} \frac{1}{2R^2}-\frac{1}{2r_0^2}-\frac{\varphi}{r_0^3}\\
\geq &- C +\int_{|s|\leq T} \left(\frac{r_0}{(r_0^2+\rho^2)^{3/2}}- \frac{G_0^2}{r_0^3} \right) \varphi \geq - C(1+\lVert \dot{\varphi} \rVert_{L^2})
\end{split}
\]
for some $C>0$ which depends only on $R$. An analogous computation shows that for the third term we have
\[
E(\varphi)\geq \int_{s\in\mathbb{R}}\frac{1}{(r_0^2+\rho^2)^{1/2}}-\frac{1}{r_0}+\left(\frac{1}{r_0^2}-\frac{(r_0+\rho)}{(r_0^2+\rho^2)^{3/2}}+\frac{(G_0^2-G^2)}{r_0^3}\right)\varphi\geq -C(1+\lVert \dot{\varphi} \rVert_{L^2} )
\]
for some $C>0$ which depends only on $R$. Therefore 
\[
\mathcal{A}_G(\varphi)\geq \frac{\lVert \dot{\varphi} \rVert^2_{L^2}}{2}-C(1+\lVert \dot{ \varphi} \rVert_{L^2})
\]
for $C$ depending only on $R$  and the result follows after enlarging $M$ (if necessary)  while keeping $R$ fixed.
\end{proof}

We now have established the existence of the mountain pass geometry for the level sets of the functional $\mathcal{A}_G$. The next natural step would be to apply the classical deformation lemma to obtain a Palais-Smale (PS) sequence for the functional $\mathcal{A}_G$. There are however two difficulties. The first one is that, a priori, a suboptimal path, might contain points $\varphi\in D^{1,2}$ for which $\min_{s\in\mathbb{R}} (r_0+\varphi)(s)=0$, at which the functional $\varphi\mapsto \mathcal{A}_G$ is not continuous. The second difficulty is that, even if we can guarantee that  $\min_{s\in\mathbb{R}} (r_0+\varphi)(s)>0$ for all $\varphi$ in the region where we carry the deformation argument, without further constraints we are not able to show that the PS sequence obtained is precompact. For that reason, we take $\overline{m}>0$ large enough and we carry the deformation argument in the region 
\begin{equation}\label{eq:constrainedset}
\mathcal{F}_{\overline{m}}=\left\{\varphi\in D^{1,2}\colon \min_{s\in\mathbb{R}}(r_0+\varphi)(s)\leq \overline{m} \right\}.
\end{equation}
In Lemma \ref{lem:suffconditionboundedness} we show that on a suitable subset $\mathcal{F}_{\overline{m},\delta,b} \subset \mathcal{F}_{\overline{m}}$, the functional $\mathcal{A}_G(\varphi)$ is bounded and coercive, from where we deduce a uniform bound for $\lVert \varphi\rVert$ when $\varphi\in\mathcal{F}_{\overline{m},\delta,b}$. This will be crucial to obtain uniformly bounded PS sequences.

\subsubsection{The deformation argument}\label{sec:deformation}

We now introduce the set of curves
\begin{equation}\label{eq:admissiblepaths}
\varGamma=\left\{\gamma  \in C([0,1],D^{1,2})\colon \gamma(0)=\psi_1,\ \gamma(1)=\psi_2 \right\}
\end{equation}
and for $\overline{m}>0$ large enough the candidate to critical value
\begin{equation}\label{eq:definitioncriticalvalue}
c_G=\inf_{\gamma\in\varGamma}\  \max \{  \mathcal{A}_G(\gamma(t))\colon \gamma(t)\in \mathcal{F}_{\overline{m}},\ t\in[0,1]\}
\end{equation}
The first step in the deformation argument is to prove that there exists a positive $\delta$ such that for all bounded $\varphi\in \{\varphi\in D^{1,2}\colon |\mathcal{A}_G-\varphi |\geq \delta\}$,  we have $\min_{s\in\mathbb{R}} (r_0+\varphi)(s)>0$. To that end we notice that 
\begin{equation}\label{eq:separationfunctionalparameter}
\mathcal{A}_G(\varphi)=A(\varphi)-G^2\  B(\varphi)
\end{equation}
with
\begin{equation}
\begin{split}
A(\varphi)=&\int_{\mathbb{R}}  \frac{\dot{\varphi}^2}{2}+\frac{1}{((r_0+\varphi)^2+\rho^2)^{1/2}}-\frac{1}{r_0}+\frac{\varphi}{r_0^2}+G_0^2 \left(\frac{1}{2r_0^2}-\frac{\varphi}{r_0^3}\right)\\
B(\varphi)=& \int_{\mathbb{R}} (r_0+\varphi)^{-2}
\end{split}
\end{equation}
and apply a monotonicity trick due to  Struwe (see \cite{MR926524} and \cite{MR1736116})  to show that for almost every $G$, the functional $B(\varphi)$ is bounded if $|\mathcal{A}_G(\varphi)-c_G|$ is small enough (see Remark \ref{rem:aprioriestimateremark}). The following version of the monotonicity trick was proved in \cite{MR1718530}. We provide the proof for the sake of self completeness.

\begin{lem}\label{lem:monotonicitylemma}
There exists a full measure subset $J\subset [-G_0,G_0]$ such that for all  $G\in J$ there exists constants $\delta>0$ and $C>0$  for which if $|\mathcal{A}_{G}(\varphi)-c_G|\leq \delta$ then $B(\varphi)\leq C$.
\end{lem}

\begin{proof}
Since $B(\varphi)\geq 0$ it follows from expression \eqref{eq:separationfunctionalparameter} and the definition of $c_G$ in \eqref{eq:definitioncriticalvalue} that $G\mapsto c_G$ is a monotone decreasing function. Therefore, it is differentiable on a subset $J\subset \mathbb{R}$ whose complement has zero measure. Let $G^*\in J$, $\delta>0$ and take $\varphi$ such that $|I_{G^*}(\varphi)-c_{G^*}|\leq \delta$. Take now $G<G^*$, then, by decreasing (if necessary) the value of $\delta$ we can assume that
\[
\mathcal{A}_G(\varphi)\geq c_{G^*}-(G^*-G)\qquad\qquad \mathcal{A}_{G^*}(\varphi)\leq c_{G^*}+(G^*-G)
\]
Then
\[
 B(\varphi)=\frac{\mathcal{A}_{G^*}(\varphi)-\mathcal{A}_{G} (\varphi)}{G^*-G}\leq \frac{c_G+(G^*-G)-c_G^*+(G^*-G)}{G^*-G}
\]
By the hypothesis on $G^*$ there exists an open neighbourhood around $G^*$ for which 
\[
-c'_{G^*}-1\leq\frac{c_G-c_{G^*}}{G^*-G}\leq -c'_{G^*}+1
\]
and the lemma is proven.
\end{proof}

Boundedness of the functional $B(\varphi)$ allows us to obtain an a priori estimate for $\min_{s\in\mathbb{R}} (r_0+\varphi)(s)$ if $\varphi\in D^{1,2}$ is bounded.

\begin{lem}\label{lem:strongforcecond}
Let $\varphi\in D^{1,2}$ be such that $B(\varphi)\leq C$. Then, there exists a constant $\underline{m}>0$, depending only on $\lVert \dot{\varphi} \rVert_{L^2}$ such that 
\[
r_0(s)+\varphi(s)\geq \underline{m} \qquad\qquad \forall s\in\mathbb{R}.
\]
\end{lem}
\begin{proof}
Suppose there exists $s_*\in\mathbb{R}$ such that $\lim_{s\to s_*}r_0(s)+\varphi(s)=0$. Since $\varphi\in D^{1,2}$ it holds that $|s_*|<\infty$ and we can assume without loss of generalitiy that $r_0(s)+\varphi(s)>0$ for all $s<s_*$. Take now $s_0=s_*-1$ and write $r(s)=r_0(s)+\varphi(s)$. Then, by the fundamental theorem of calculus, for any $s\in[s_0,s_*)$
\[
\ln(r(s))-\ln(r(s_0))=  \int_{r(s_0)}^{r(s)} r^{-1} \mathrm{d}r=\int_{s_0}^s r^{-1}(t) \dot{r}(t)\mathrm{d}t
\]
and H\"older's inequality implies
\[
| \ln(r(s))-\ln(r(s_0))|\leq B(\varphi)  \left( \int_{s_0}^s(\dot{r}_0+\dot{\varphi})^2 \right)^{1/2}\leq C \left(1+\lVert \dot{\varphi} \rVert_{L^2}\right).
\]
\end{proof}

\begin{rem}\label{rem:aprioriestimateremark}
In Lemma \ref{lem:homoclinicsnotcollision}, we show that for all $G\in\mathbb{R}\setminus\{0\}$, we have $\min_{s\in\mathbb{R}}(r_0+\varphi)(s)\geq G^2/2$ for orbits of \eqref{eq:oneparameterHamiltonian} which are homoclinic to $\gamma_\infty$. However, that argument does not allow us to conclude that there exist $\delta>0$ such that for all bounded  $\varphi\in \{\varphi\in D^{1,2}\colon |\mathcal{A}_G-\varphi |\geq \delta\}$ we have $\min_{s\in\mathbb{R}}(r_0+\varphi)(s)>0$. Therefore, is not clear how to incorporate the a priori estimate in Lemma \ref{lem:homoclinicsnotcollision} to obtain a minmax critical point.
\end{rem}

Assume now that $\varphi$ is such that $|\mathcal{A}_G(\varphi)-c_G|\leq\delta$. Therefore, thanks to Lemmas \ref{lem:monotonicitylemma} and \ref{lem:strongforcecond} it is possible to obtain an inequality of the form
\begin{equation}\label{eq:coercivityheuristics}
 \mathcal{A}_G(\varphi) \geq \frac{\lVert \dot{\varphi}\rVert_{L^2}}{2}-C \lVert \varphi\rVert
\end{equation}
for some $C>0$. Thus, if we moreover assume that $\varphi\in \mathcal{F}_{\overline{m}}$ and that $\inf_{s\in\mathbb{R}} (r_0+\varphi)(s)$ happens for $s\in[0,1]$ we can obtain a uniform bound for the $D^{1,2}$ norm of $\varphi$. In general, in problems in which the action functional is invariant under integer time translations, the latter assumption introduces no loss of generality and this argument can be employed to obtain uniformly bounded PS sequences.

However, in the present problem, the translation operator $T_\tau(\varphi)=\varphi(s+\tau)+r_0(s+\tau)-r_0(s)$, for which we have $\mathcal{A}_G(T_\tau(\varphi))=\mathcal{A}_G(\varphi)$, is not an isometry in $D^{1,2}$. This introduces certain technicalities in the deformation argument. In order to overcome this technical annoyance we introduce the following definition.

\begin{defn}\label{defn:barycenter}
Given $\varphi\in D^{1,2}$ we define its \textit{barycenter} as the functional $\mathrm{Bar}:D^{1,2}\to\mathbb{R}$ given by
\[
\mathrm{Bar}(\varphi)=\left( \int_{\mathbb{R}}  (1+(r_0+\varphi)^2)^{-2}\mathrm{d}s \right)^{-1} \int_{\mathbb{R}} s (1+(r_0+\varphi)^2)^{-2}\mathrm{d}s.
\]
\end{defn}

The following properties of the barycenter functional will be crucial for the deformation argument.

\begin{lem}\label{lem:propertiesbarycenter}
Let $\mathrm{Bar}(\varphi)$ be the functional introduced in Definition \ref{defn:barycenter}. The following statements hold:
\begin{itemize}
\item Behaviour under translations: For any $\tau\in \mathbb{Z}$ 
\[
B(T_{\tau} \varphi)=B(\varphi)-\tau.
\]
where the translation operator $T_\tau$ was introduced in Lemma \ref{lem:periodicitylemma}.
\item Local Lipschitzianity: For any $K>0$ there exists $L_{\mathrm{Bar}}>0$ such that 
\[
\sup_{\lVert\varphi\rVert\leq K, \lVert\varphi'\rVert\leq K} \frac{|\mathrm{Bar}(\varphi)-\mathrm{Bar}(\varphi')|}{\lVert \varphi-\varphi' \rVert}\leq L_{\mathrm{Bar}}.
\]
\end{itemize}

\end{lem}

\begin{proof}
The proof of the first part is a trivial computation. For the second one we express
\[
B(\varphi)= B_2(\varphi)/B_1(\varphi)
\]
with 
\[
B_1(\varphi)=\int_{\mathbb{R}} (1+(r_0+\varphi)^2)^{-2}\mathrm{d}s\qquad\qquad  B_2(\varphi)=\int_{\mathbb{R}} s(1+(r_0+\varphi)^2)^{-2}\mathrm{d}s.
\]
First we notice that there exists $C>0$ such that for all $\lVert \varphi\Vert\leq K$ we have $B_1(\varphi)\geq C>0$ and  $|B_2(\varphi)|\leq C$. Indeed, for all $s\in\mathbb{R}$ 
\[
|\varphi(s)|\leq |\varphi(0)|+ \lVert \dot{\varphi} \rVert_{L^2}\sqrt{|s|}\leq (1+\sqrt{|s|})\lVert \varphi \rVert \leq C (1+\sqrt{|s|})
\]
so there exists $T>0$, depending only on $\lVert \varphi\rVert$, such that 
\[
r_0(s)+\varphi(s) \geq  r_0(s) (1-\mathcal{O} (s^{-1/6}))
\]
for all $|s|\geq T$. Therefore, for $C$ depending only on $T$,
\[
|B_2(\varphi)|\leq C+ \int_{|s|\geq T} |s| (1+(r_0+\varphi)^2)^{-2} \mathrm{d}s\leq C\left(1+\int_{|s|\geq T} |s| r_0^{-4}(s) \mathrm{d}s \right)\leq C.
\]
The uniform bound from below for $B_1(\varphi)$ follows since there exists $C$ and $T$, depending only on $\lVert \varphi\rVert$, such that 
\[
r_0(s)+\varphi(s)\leq C (1+\sqrt{|s|})
\]
for all $|s|\leq T$. Take now $\varphi,\varphi^*\in D^{1,2}$ and write
\[
B(\varphi^*)-B(\varphi)=(B_2(\varphi^*)-B_2(\varphi))/B_1(\varphi)+  (B_1(\varphi^*)-B_1(\varphi))B_2(\varphi) / B_1(\varphi^*)B_1(\varphi)
\]
Let $g(\varphi)=(1+(r_0+\varphi)^2)^{-2}$. Then for $\varphi,\varphi^*\in D^{1,2}$ we can write
\[
B_2(\varphi^*)-B_2(\varphi)=\int_{\mathbb{R}} s (\varphi^*-\varphi) \int_{0}^1 g'( \lambda(\varphi^*-\varphi)) \mathrm{d}\lambda\mathrm{d}s
\]
On one hand for all $s\in\mathbb{R}$
\[
|\varphi^*(s)-\varphi(s)|\leq |\varphi^*(0)-\varphi(0)|+\lVert \dot{\varphi}^*-\dot{\varphi} \rVert _{L^2} \sqrt{|s|}\leq (1+\sqrt{|s|}) \lVert \varphi^*-\varphi \rVert
\]
and it follows that 
\[
\begin{split}
|B_2(\varphi^*)-B_2(\varphi)|\leq& C \lVert \varphi^*-\varphi\rVert \int_{\mathbb{R}} (1+|s|^{3/2}) \int_{0}^1 |r_0+\lambda(\varphi^*-\varphi)|(1+|r_0+\lambda\varphi^*-\varphi|^2)^{-3} \mathrm{d}\lambda\mathrm{d}s\\
\leq & C \lVert \varphi^*-\varphi\rVert \int_\mathbb{R} (1+|s|^{3/2}) r_0^{-5} \leq C \lVert \varphi^*-\varphi\rVert.
\end{split}
\]
The same computation shows that there exists $C$ such that 
\[
|B_1(\varphi^*)-B_1(\varphi)|\leq C \lVert \varphi^*-\varphi\rVert
\]
and the lemma is proven.
\end{proof}

Together, Lemma \ref{lem:monotonicitylemma} and Lemma \ref{lem:propertiesbarycenter} imply the following result, which is key in our constrained deformation argument.

\begin{lem}\label{lem:suffconditionboundedness}
Let $J\subset[-G_0,G_0]\subset \mathbb{R}$ be the subset obtained in Lemma \ref{lem:monotonicitylemma}. Let $G\in J$, let $\delta>0$ be the constant in Lemma \ref{lem:monotonicitylemma},  let $b>0$ and define
\begin{equation}
\mathcal{F}_{\overline{m},\delta,b}=\mathcal{F}_{\overline{m}} \cap \left\{\varphi\in D^{1,2}\colon |\mathcal{A}_G(\varphi)-c_G|\leq \delta,\ \  |\mathrm{Bar} (\varphi)| \leq b \right\}
\end{equation}
Then,   there exists $K>0$ such that 
\[
\sup_{\varphi\in\mathcal{F}_{\overline{m},\delta,b}}\lVert \varphi \rVert\leq K
\]
\end{lem}

\begin{proof}
Let 
\[
\mathcal{S}=\{\bar{s}\in\mathbb{R}\colon\min_{s\in\mathbb{R}} (r_0+\varphi)(s)= (r_0+\varphi)(\bar{s})\}
\]
We claim that under the hypothesis of the lemma there exists $C>0$  such that
\[
 |\mathrm{Bar}(\varphi)-\bar{s}|\leq C\qquad\qquad \forall\bar{s}\in \mathcal{S}
\]
Suppose not, then,  by continuity, there exist $\bar{s}\in\mathcal{S}$ and sequences $\{\varphi_n\}$, $\{\bar{s}_n\}$ such that $\varphi_n\to \varphi$ in $D^{1,2}$,  $\bar{s}_n\to \bar{s}$ and
\[
\varphi_n\in \mathcal{F}_{\overline{m}},\qquad\qquad
|\mathcal{A}_G(\varphi_n)-c_G|\leq\delta,\qquad\qquad\text{and}\qquad\qquad |\mathrm{Bar}(\varphi_n)-\bar{s}_n|\to\infty.
\]
 By invariance of the action functional under translation $T_\tau(\varphi)(s)=\varphi(s+\tau) +r_0(s+\tau)-r_0(s)$ (see Lemma \ref{lem:periodicitylemma}) and the fact that 
\[
B(T_\tau (\varphi))=B(\varphi)-\tau
\]
the sequence  $\tilde{\varphi}_n=T_{[\bar{s}_n]} \varphi_n$, satisfies 
\[
\tilde{\varphi}_n\in \mathcal{F}_{\overline{m}},\qquad\qquad|\mathcal{A}_G(\tilde{\varphi}_n)-c_G|\leq \delta\qquad\qquad\text{and}\qquad\qquad|\mathrm{Bar}(\tilde{\varphi}_n)|\to \infty.
\]
 However, the first two properties imply the existence of $K$ such $\lVert \tilde{\varphi}_n \rVert\leq K$ for all $n\in\mathbb{N}$. Indeed, by the construction of $\tilde{\varphi}_n$, for all  $\tilde{\varphi}_n$ we have
\[
|\varphi(s)|\leq \overline{m}+\lVert \dot{\tilde{\varphi}}_n \rVert_{L^2}\sqrt{|s+1|}
\]
Moreover, since $G\in J$ and $|\mathcal{A}_G(\tilde{\varphi}_n)-c_G|\leq \delta$ ,  Lemma \ref{lem:monotonicitylemma} implies that there exists $C>0$ such that $B(\tilde{\varphi}_n)\leq C$ for all $n\in\mathbb{N}$. Therefore, it is easy to check that there exists $C>0$ such that 
\[
\mathcal{A}_G(\tilde{\varphi}_n)\geq \frac{\lVert \dot{\tilde{\varphi}}_n \rVert^2_{L^2}}{2}-C (1+\lVert \dot{\tilde{\varphi}}_n \rVert_{L^2})
\]
for all $n\in\mathbb{N}$. Thus, since $|\mathcal{A}_G(\tilde{\varphi}_n)-c_G|\leq \delta$ and $|c_G|$ is bounded, the sequence $\{\tilde{\varphi}_n\}$ must be uniformly bounded.  It is easy to check that the existence of  $K$ such that $\lVert \tilde{\varphi}_n \rVert\leq K$  is in contradition with $\mathrm{Bar}(\tilde{\varphi}_n)$ being unbounded.

Once we know that the claim $ |\mathrm{Bar}(\varphi)-\bar{s}|\leq C$ holds, we obtain that $\bar{s}\leq C+b$ for all $\bar{s}\in\mathcal{S}$. Therefore, since $\varphi\in \mathcal{F}_{\overline{m}}$, we have that 
\[
|\varphi(0)|\leq |\varphi(\bar{s})|+\lVert \dot{\varphi} \rVert _{L^2} \sqrt{|\bar{s}|}\leq C(1+\lVert \dot{\varphi} \rVert _{L^2} )
\]
for some $C>0$ depending only on $\overline{m},\delta$ and $b$. The result follows since now we can show that for all $\varphi_n$ we have 
\[
\mathcal{A}_G(\varphi_n)\geq \frac{\lVert \dot{\varphi}_n \rVert^2_{L^2}}{2}-C (1+\lVert \dot{\varphi}_n \rVert_{L^2})
\] 
for some $C$ uniform on $n$.
\end{proof}

In Proposition \ref{prop:existencePS} we show the existence of a PS sequence contained in $\mathcal{F}_{\overline{m},\delta,b}$ for large enough values of $\overline{m}$ and $b$. Notice that in particular, thanks to Lemma \ref{lem:suffconditionboundedness} this sequence will be uniformly  bounded.

We split the proof of Proposition \ref{prop:existencePS} in two parts. First, we assume by contradiction that there exists no critical point of the action functional $\mathcal{A}_G$ in  $\mathcal{F}_{\overline{m},\delta,b}$. Under this assumption, we build a pseudo gradient vector field for $\mathcal{A}_G$, this is the content of Proposition \ref{prop:pseudograd}. Then, in Proposition \ref{prop:existencePS} we use this pseudo gradient vector field to build a localized deformation which yields points $\varphi\in \mathcal{F}_{\overline{m},\delta,b}$ for which $\mathcal{A}_G(\varphi)<c_G$, a contradiction.

Before stating Proposition \ref{prop:pseudograd} some definitions are in order. Let $b>0$ and $0<\varepsilon<\delta/2$ where $\delta>0$ is the constant in Lemma \ref{lem:monotonicitylemma}. For we want the flow along the pseudogradient vector field to leave $D^{1,2}\setminus \mathcal{F}_{\overline{m}}$ positively invariant, we express it as the convex combination of two localized vector fields: a gradient-like vector field supported on 
\begin{equation}\label{eq:defnP}
P= \{\varphi\in D^{1,2}\colon |\mathcal{A}_G(\varphi)-c_G|\leq \varepsilon,\ \min_{s\in\mathbb{R}} r_0(s)+\varphi(s)\leq \overline{m},\ \mathrm{Bar}(\varphi)\leq 2 b\}
\end{equation}
and a vector field supported on
\begin{equation}\label{eq:dfnY}
Y= \{\varphi\in D^{1,2}\colon |\mathcal{A}_G(\varphi)-c_G|\leq \delta,\ |\min_{s\in\mathbb{R}} r_0(s)+\varphi(s)-\overline{m}|\leq \varepsilon,\ \mathrm{Bar}(\varphi)\leq 2 b\}.
\end{equation}
for which $D^{1,2}\setminus \mathcal{F}_{\overline{m}}$ is positively invariant. This construction is made explicit in the following proposition.

\begin{prop}\label{prop:pseudograd}
Let $J\subset[G_0,G_0]\mathbb{R}$ be the  subset obtained in Lemma \ref{lem:monotonicitylemma}. Let $\delta>0$ be the constant in Lemma \ref{lem:monotonicitylemma} and take $\overline{m}>0$ large enough. Assume that for  all $b>0$ there exists $\alpha>0$ for which
\[
\inf \left\{ \lVert \nabla \mathcal{A}_G(\varphi)\rVert \colon \varphi\in\mathcal{F}_{\overline{m},\delta,2b} \right\}\geq \alpha
\]
Then, there exists $b_0$ such that for all $b\geq  b_0$  there exists a Lipschitz pseudogradient vector field $W:D^{1,2}\to D^{1,2}$ such that
\begin{itemize}
\item $ \lVert W \rVert \leq 1$,
\item There exists a constant $\beta>0$ such that 
\[
\mathrm{d}\mathcal{A}_G (\varphi)\left[W(\varphi)\right] \leq -\beta\qquad \qquad \forall \varphi\in P\cup Y,
\]
\item The region $D^{1,2}\setminus \mathcal{F}_{\overline{m}}$ is positively invariant under the flow of $W$.
\end{itemize}

\end{prop}

\begin{proof}
Let $\varepsilon<\delta/2$, define the sets
\[
\begin{split}
Y= &\{\varphi\in D^{1,2}\colon |\mathcal{A}_G(\varphi)-c_G|\leq \delta,\ |\min_{s\in\mathbb{R}} r_0(s)+\varphi(s)-\overline{m}|\leq \varepsilon,\ \mathrm{Bar}(\varphi)\leq 2 b\}\\
Z=&\{\varphi\in D^{1,2}\colon |\mathcal{A}_G(\varphi)-c_G|\leq \delta,\ |\min_{s\in\mathbb{R}} r_0(s)+\varphi(s)-\overline{m}|\geq 2\varepsilon,\ \mathrm{Bar}(\varphi)\leq 2b\}
\end{split}
\]
and the function
\[
\Psi=\frac{\mathrm{dist}Z}{\mathrm{dist}Y+\mathrm{dist} Z},
\]
which satisfies $\Psi=0$ on $Z$ and $\Psi=1$ on $Y$. We also introduce
\[
\begin{split}
P= &\{\varphi\in D^{1,2}\colon |\mathcal{A}_G(\varphi)-c_G|\leq \varepsilon,\ \min_{s\in\mathbb{R}} r_0(s)+\varphi(s)\leq \overline{m},\ \mathrm{Bar}(\varphi)\leq 2 b\}\\
Q= &\{\varphi\in D^{1,2}\colon |\mathcal{A}_G(\varphi)-c_G|\geq \delta,\  \min_{s\in\mathbb{R}} r_0(s)+\varphi(s)\leq \overline{m},\ \mathrm{Bar}(\varphi)\leq 2 b\}\\
\end{split}
\]
and define the function 
\[
\Phi=\frac{\mathrm{dist}Q}{\mathrm{dist}P+\mathrm{dist} Q}.
\]
which satisfies $\Phi=0$ on $Q$ and $\Phi=1$ on $P$. Take now a sufficiently small open neighbourhood $U\subset D^{1,2}$, $\mathrm{supp}(\Phi)\subset D^{1,2}$. Notice that by Lemma \ref{lem:suffconditionboundedness}, there exists $K>0$ such that 
\[
\sup \{\lVert \varphi\rVert \colon\varphi\in U\}\leq K.
\]
Then, since $G\in\mathcal{J}$,  Lemmas \ref{lem:monotonicitylemma} and \ref{lem:strongforcecond}, impy that there exists $\underline{m}>0$ such that  
\[
\inf\{ \min_{s\in\mathbb{R}} r_0(s)+\varphi(s)\colon \varphi\in U\}\geq \underline{m}.
\]
Therefore, by Lemma \ref{lem:C1lemma} we have that $\mathrm{d}\mathcal{A}_G\in C^1(U,D^{1,2})$, what implies the existence of a constant $C>0$ such that  $\mathrm{dist}P+\mathrm{dist} Q>C$. We introduce now the pseudogradient vector field
\begin{equation}\label{eq:pseudgrad}
W=\frac{1}{ \sqrt{2}}(- (1-\Psi)\Phi \frac{\nabla \mathcal{A}_G}{\lVert \nabla \mathcal{A}_G\rVert}+\Psi v )
\end{equation}
where $v$ is the constant vector field given by the constant $v=1\in D^{1,2}$. Notice that for a  large enough fixed $\overline{m}$, and for all $b>0$ there exists $\tilde{\alpha}>0$ such that
\begin{equation}\label{eq:constantv}
\sup \{ \mathrm{d}\mathcal{A}_G(\varphi)[v] \colon \varphi\in \mathrm{supp} \Psi \}\leq -\tilde{\alpha}
\end{equation}
Indeed
\[
\mathrm{d}\mathcal{A}_G(\varphi)[v]=\int_{\mathbb{R}} \left( \frac{G^2}{(r_0+\varphi)^3}-\frac{r_0+\varphi}{((r_0+\varphi)^2+\rho^2)^{3/2}}\right)v
\]
and the claim follows since for large enough $\overline{m}$ the integrand is non positive and moreover it is strictly negative on a positive measure subset of the real line since (thanks to Lemma \ref{lem:suffconditionboundedness}) there exists $K>0$ such that $\sup_{\varphi\in \mathrm{supp} \Psi} \lVert \varphi\rVert \leq K$. It is straightforward to chek that the pseudogradient vector field $W$ introduced in \eqref{eq:pseudgrad} satisfies the properties listed in the statement in the lemma with $\beta=\min\{\alpha,\tilde{\alpha}\}>0$.

\end{proof}

\begin{prop}\label{prop:existencePS}
Let $J\subset[-G_0,G_0]\subset\mathbb{R}$ be the subset obtained in Lemma \ref{lem:monotonicitylemma}. Let $\delta>0$ be the constant in Lemma \ref{lem:monotonicitylemma} and take $\overline{m}>0$ large enough. Then, for $b>0$ large enough there exists a Palais-Smale sequence $\left\{\varphi_n\right\}_{n\in\mathbb{N}}\subset \mathcal{F}_{\overline{m},\delta,2b}$ for $\mathcal{A}_G$ at the level $c_G$.
\end{prop}

\begin{proof}

Let $\varGamma\subset C([0,1],D^{1,2})$ be the set  defined in \eqref{eq:admissiblepaths}. Let $\varepsilon<\delta/2$ and take a suboptimal path $\gamma_\varepsilon\subset \varGamma$ for which $\mathcal{A}_G(\gamma_\varepsilon(t))\leq c_G+\varepsilon$ for all $t\in [0,1]$ such that $\min_{s\in\mathbb{R}} r_0(s)+(\gamma_\varepsilon(t))(s)\leq \overline{m}$.  For all $\gamma\in \varGamma$ we define the set
\[
\mathcal{B}_{\gamma}\equiv \left\{ \mathrm{Bar} (\gamma_\varepsilon(t))\colon |\mathcal{A}_G(\gamma(t))-c_G|\geq \varepsilon,\ \  \min_{s\in\mathbb{R}} r_0(s)+(\gamma(t))(s)\leq \overline{m},\ \ t\in[0,1]  \right\}.
\]
Notice that for each $\gamma\in \varGamma$, the set $\mathcal{B}_\gamma$ is a compact subset of $\mathbb{R}$.  Denote by 
\[
\begin{split}
b_{\mathrm{min}}=\min \mathcal{B}_{\gamma_\varepsilon}\qquad\qquad b_{\mathrm{max}}= \max  \mathcal{B}_{\gamma_\varepsilon}
\end{split}
\]
and consider the translated path $\gamma_\varepsilon^1=T_{b_0} \gamma_\varepsilon$ for $b_0=[b_{\mathrm{min}}]$. It satisfies that 
\[
\mathcal{B}_{\gamma_\varepsilon^1}\subset[0,b_{\mathrm{max}}-b_0+1]
\]
 Let $W:D^{1,2}\to D^{1,2}$ be the pseudogradient vector field built in Proposition \ref{prop:pseudograd} and denote by $\eta_\tau$ its time $\tau$ flow. Notice that since $W$ is Lipschitz the flow $\eta_\tau$ is well defined at least for sufficiently small $\tau$. Let $\beta>0$ be the constant in Proposition \ref{prop:pseudograd}. We claim that the deformed path $\gamma_\varepsilon^1 \circ \eta_{\tau^*}$ with $\tau_*= 2 \varepsilon/\beta$ satisfies  
\[
\begin{split}
\max  \{ & \mathcal{A}_G(\eta_{\tau_*}(\gamma_\varepsilon^1) (t))\colon  \min_{s\in\mathbb{R}} r_0(s)+\left(\eta_{\tau_*}(\gamma_\varepsilon^1 ) (t)\right) (s)\leq \overline{m},\ |\mathrm{Bar}(\eta_{\tau_*}(\gamma_\varepsilon^1 ) (t))| \leq b, \ t\in[0,1]  \} \\
& \leq c_G-\varepsilon.
\end{split}
\]
To verify the claim we first notice that the maximal displacement is bounded by
\[
\lVert \eta_{\tau^*} (\varphi)-\varphi\rVert \leq \tau^* \lVert W \rVert\leq \tau^*=2\varepsilon/\beta.
\]
Therefore, applying Lemma \ref{lem:propertiesbarycenter}, we obtain that for any $\varphi\in \gamma_\varepsilon$ (taking $b$ sufficiently large)
\[
 |\mathrm{Bar}(\eta_{\tau^*} (\varphi))-\mathrm{Bar}(\varphi)|\leq L_{\mathrm{Bar}} 2\varepsilon/ \beta \leq b/4.
\]
Thus, since  the region $\{ \min_{s\in\mathbb{R}} r_0(s)+\varphi(s)\geq \overline{m}\}$ is forward invariant by the flow $\eta_\tau$ and 
\[
\frac{\mathrm{d}}{\mathrm{d}\tau}( \mathcal{A}_G \circ \eta)\leq 0,
\]
 in order to verify the claim, it is enough to check that there does not exist
\[
\varphi\in \left\{  \min_{s\in\mathbb{R}} r_0(s)+\left(\gamma_\varepsilon^1  (t)\right) (s)\leq  \overline{m},\ \   \mathrm{Bar}(\gamma_\varepsilon^1  (t))\leq 5b/4, \ \ t\in[0,1]  \right\} 
\]
for which $\eta_\tau (\gamma_{\varepsilon}^1)\in P\ \  \forall \tau\in[0,\tau^*]$ where $P\subset D^{1,2}$ is the set defined in \eqref{eq:defnP}. This is clearly not possible since for $\varphi\in P$ we have 
\[
\frac{\mathrm{d}}{\mathrm{d}\tau}( \mathcal{A}_G \circ \eta)\leq -\beta
\]
so 
\[
\mathcal{A}_G(\eta_{\tau^*}(\gamma_\varepsilon^1))\leq \mathcal{A}_G(\gamma_\varepsilon^1)-\tau^*\beta \leq c_G-\varepsilon
\]
a contradiction. Now that the claim is verified consider the path
\[
\gamma_{\varepsilon}^2=T_{-b_0} (\gamma_\varepsilon^1).
\]
It satisfies that 
\[
\mathcal{B}_{\gamma_\varepsilon^2} \subset[b_0+b, b_{\mathrm{max}}+1]
\]
and
\[
\max \{  \mathcal{A}_G(\gamma_\varepsilon^2 (t))\colon  \min_{s\in\mathbb{R}} r_0(s)+\left(\gamma_\varepsilon^2 (t)\right) (s)\leq \overline{m},\ \  \mathrm{Bar}(\gamma_\varepsilon^2 (t)) \leq b_0+ b, \ \ t\in[0,1]  \}  \leq c_G-\varepsilon 
\]
If $b_{\mathrm{max}}-b_0+1\leq b$ the proposition is proved. In the case $b_{\mathrm{max}}-b_0 +1\geq b$ we repeat the argument above with the path $\gamma_\varepsilon^2$ to obtain a path $\gamma_\varepsilon^3$ satisfying 
\[
\mathcal{B}_{\gamma_\varepsilon^3} \subset[b_0+2b, b_{\mathrm{max}}+1]
\]
and
\[
\max \{  \mathcal{A}_G(\gamma_\varepsilon^3 (t))\colon  \min_{s\in\mathbb{R}} r_0(s)+\left(\gamma_\varepsilon^3 (t)\right) (s)\leq \overline{m},\ \  \mathrm{Bar}(\gamma_\varepsilon^3 (t)) \leq b_0+2b, \ \ t\in[0,1]  \}  \leq c_G-\varepsilon 
\]
The result follows after repeating the construction no more than $[(b_{\mathrm{max}}-b_0+1)/b]$ steps.
\end{proof}

Finally, we obtain the existence of a critical point  of the functional $\mathcal{A}_G$ at a level $c_G$ .

\begin{thm}\label{thm:convergencePS}
Let $J\subset[-G_0,G_0]\subset \mathbb{R}$ be the subset obtained in Lemma \ref{lem:monotonicitylemma}. Then, for all $G\in J$ there exists a critical point of the action functional $\mathcal{A}_G$ at the level $c_G$.
\end{thm}

\begin{proof}
By Proposition \ref{prop:existencePS}, for sufficiently large $\overline{m}$ and $b$ there exists a Palais-Smale sequence $\left\{\varphi_n\right\}_{n\in\mathbb{N}}\subset \mathcal{F}_{\overline{m},\delta,2b}$  so it follows from  Lemma \ref{lem:suffconditionboundedness} that it is bounded. The theorem is then proved since the Palais Smale sequence $\{\varphi_n\}_{n\in\mathbb{N}}$ satisfies the hypothesis for the set $Q$ of the compactness Lemma \ref{lem:propermap}.
\end{proof}

\section{Topological transversality between the stable and unstable manifolds}\label{sec:toptransvers}

For the choice of $G_0>0$ was arbitrary, Theorem \ref{thm:convergencePS} implies that for any compact subset $[-G_0,G_0]$ of the real line, there exists a full measure subset $J\subset[-G_0,G_0]$ such that for all $G\in J$ there exists an orbit of \eqref{eq:Hamiltonian} which is homoclinic to $\gamma_\infty$. Another way of rephrasing Theorem \ref{thm:convergencePS} is that the invariant manifolds manifolds $W^\pm(\gamma_\infty,G)$ defined in \eqref{eq:globalmanifolds} intersect for almost all values of $G$ in $[-G_0,G_0]$. However, Theorem \ref{thm:convergencePS} contains no information about the geometry of the intersection, in particular wether it is transversal or not.  

Theorem \ref{thm:transversalityjde}, proved in \cite{guardia2021symbolic}, shows that the intersection between $W^\pm(\gamma_\infty,G)$ is transverse for all $G$ sufficiently large. Moreover, the local stable manifolds $W^{\pm}_{\mathrm{loc}}(\gamma_\infty;G)$ (see \eqref{eq:localinvariantmanifolds}) depend analytically on $r$ and $G$ (see Proposition \ref{prop:generatingfunctions}).  We want to exploit this facts to deduce that    the intersection of the manifolds $W^{\pm}(\gamma_\infty,G)$ (which, by Theorem   \ref{thm:convergencePS} we already know that exists for almost all values of $G\in[-G_0,G_0]$) must be topologically transverse for almost all values of $G$ in  $[-G_0,G_0]$. 

\begin{rem}
 In the following we fix a sufficiently large value of $G_0$ and work with $G\in J\subset[-G_0,G_0]$.
\end{rem}

The first step is to obtain an a priori estimate from below for $\min_{s\in\mathbb{R}} r_h(s)$.

\begin{lem}\label{lem:homoclinicsnotcollision}
Let $G\in\mathbb{R}$ and let $r_h(s;G):\mathbb{R}\to\mathbb{R}$ be an orbit of of the Hamiltonian $H_G$ in \eqref{eq:Hamiltonian} which is homoclinic to $\gamma_\infty$. Then, for all $s\in\mathbb{R}$ we have
\[
r_h(s)\geq \frac{G^2}{2}.
\]
\end{lem}

\begin{proof}
Since  $r_h(s;G):\mathbb{R}\to\mathbb{R}$ is an orbit of of the Hamiltonian $H_G$ we have that 
\begin{equation}\label{eq:secondorderinequality}
\ddot{r}_h=\frac{G^2}{r_h^3}-\frac{r_h}{(r_h^2+\rho^2)^{3/2}}\geq \frac{G^2}{r_h^3}-\frac{1}{r_h^2}.
\end{equation}
Let now $I\subset\mathbb{R}$ be an interval in which $\dot{r}_h(s)\leq0$ for all $s\in I$. Then, multiplying \eqref{eq:secondorderinequality} by $\dot{r}_h$, for all $s\in I$ we obtain 
\[
\frac{\mathrm{d}}{\mathrm{d}s}\left(\frac{\dot{r}_h^2}{2}+\frac{G^2}{2r_h^2}-\frac{1}{r_h}\right)\leq 0,
\]
that is, the energy 
\[
E(s)=\frac{\dot{r}_h^2(s)}{2}+\frac{G^2}{2r_h^2(s)}-\frac{1}{r_h(s)}
\]
is non increasing on the interval $I$. Let $I$ be a maximal interval in which $r_h(s)$ is decreasing: we distinguish between two alternatives, either 
\begin{itemize}
\item $I=(-\infty,s_1]$,  and $\lim_{s\to-\infty}r_h(s)=\infty$, $\lim_{s\to-\infty}\dot{r}_h(s)=0$ and $\dot{r}_h(s_1)=0$, or
\item $I=[s_0,s_1]$ and $\dot{r}_h(s_0)=0$, $\ddot{r}_h(s_0)\leq 0$  and $\dot{r}_h(s_1)=0$
\end{itemize}
 for some $-\infty<s_0< s_1<\infty$. In the first case we have 
\[
\lim_{s\to-\infty} E(s)=0.
\]
In the second case, using that $\dot{r}_h(s_0)=0$, $\ddot{r}_h(s_0)\leq 0$ and the inequality \eqref{eq:secondorderinequality}, we obtain
\[
E(s_0)=\frac{\dot{r}_h^2(s_0)}{2}+\frac{G^2}{2r_h^2(s_0)}-\frac{1}{r_h(s_0)} = \frac{G^2}{2r_h^2(s_0)}-\frac{1}{r_h(s_0)}\leq \frac{G^2}{r_h^2(s_0)}-\frac{1}{r_h(s_0)}\leq 0.
\]
Therefore, in both cases, for all $s\in I$ we have $E(s)\leq 0$, which implies that
\[
r(s)\geq r(s_1)\geq \frac{G^2}{2}.
\]
\end{proof}

Lemma \ref{lem:homoclinicsnotcollision} implies that for $G\neq 0$, homoclinic orbits do not intersect the section $\{r=0\}$. This fact allows us to exploit the analytic dependence of the Hamiltonian $H_G$ in the parameter $G$ to prove the following result.

\begin{lem}\label{lem:coincidencemanifolds}
The set of values of $G\in \mathbb{R}\setminus \{0\}$ for which  $W^-(\gamma_\infty,G)=W^+(\gamma_\infty,G)$, is finite.
\end{lem}

\begin{proof}
Fix any $\delta>0$ and let  $G_*$ be the constant in Theorem \ref{thm:transversalityjde}  and let $1\ll R_1<R_2$ be such that for all $G\in [-2G_*,2G_*]$ the generating function  $S^+(r,t;G)$ associated with the local stable manifold (see \ref{prop:generatingfunctions})  is well defined for all $(r,t)\in[R_1,R_2]\times\mathbb{T}$. Define the set
\[
Q=\{(r,y,t)\in \mathbb{R}_+\times\mathbb{R}\times\mathbb{T}\colon r\in(R_1,R_2),\  y>0,\  t=0\}.
\]
Whenever it exists, denote by $\gamma_G^-\subset Q\cap \mathcal{W}^u(\gamma_\infty;G)$ the connected component of $Q\cap \mathcal{W}^u(\gamma_\infty;G)$ associated with the first backwards intersection of $\mathcal{W}^u(\gamma_\infty;G)$ with $Q$ (see Figure \ref{fig:lemmanoncoincidence}). Define now the set 
\[
\mathcal{\widetilde{G}}=\{G\in \mathbb{R} \colon \delta\leq |G|\leq 2G_*,\  \gamma_G^-\neq\emptyset\ \text{and}\  \exists\  \varphi^-_G\in C^\omega\left([R_1,R_2],\mathbb{R}\right)\ \text{such that}\ \gamma_G^{-}=\mathrm{graph}(\varphi^-_G)\}.
\]
\begin{figure}\label{fig:lemmanoncoincidence}
\centering
\includegraphics[scale=0.60]{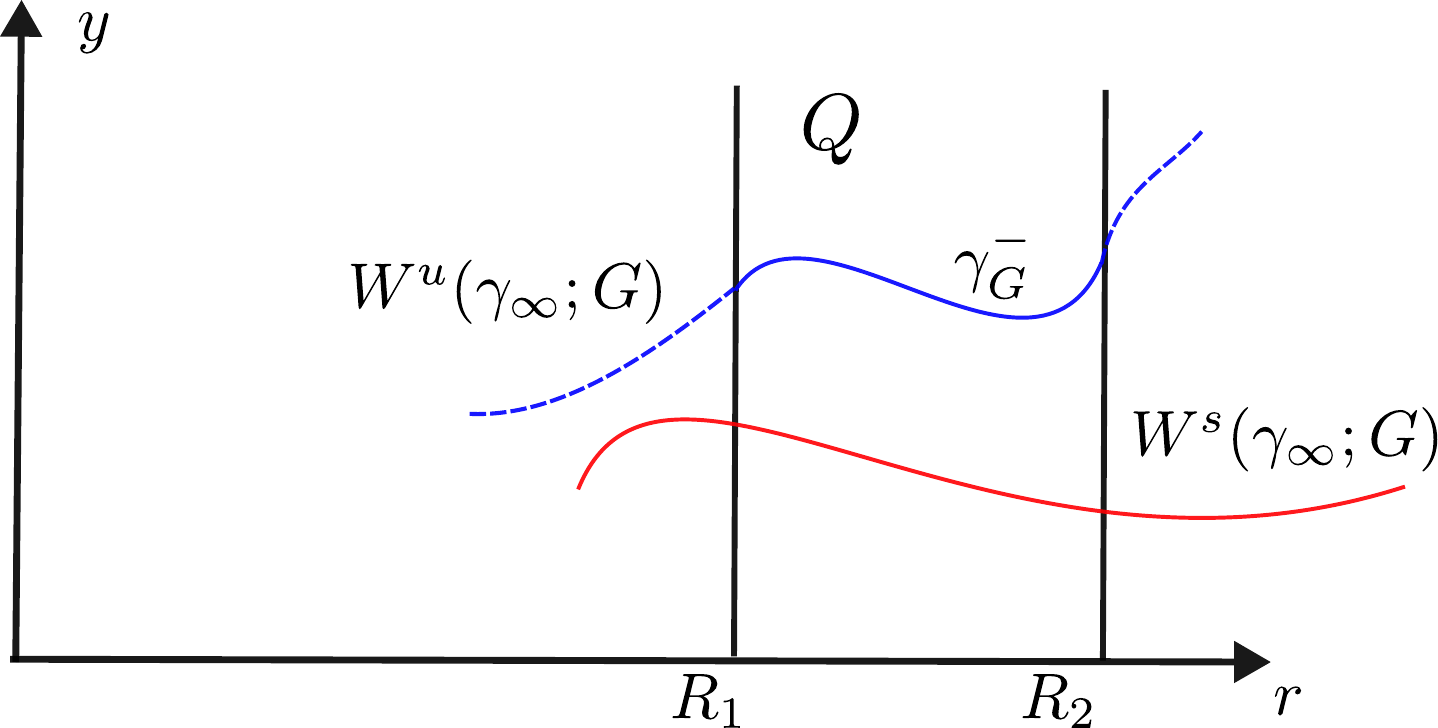}
\caption{The domain $Q$ and a sketch of the intersection of the stable manifolds $W^\pm(\gamma_\infty;G)$ with $Q$ for a value of $G\in\mathcal{\widetilde{G}}$.}
\end{figure}

\noindent Clearly, $\mathcal{G}\subset \mathcal{\widetilde{G}}$ where 
\[
\mathcal{G}=\{G\in\mathbb{R} \colon \delta\leq |G|\leq 2G_*,\  W^+(\gamma_\infty;G)=W^-(\gamma_\infty;G)\}.
\]
In view of Lemma \ref{lem:homoclinicsnotcollision}, for all $G\in\mathcal{G}$
\begin{equation}\label{eq:distancehomocliniczero}
\mathrm{dist}(W^\pm(\gamma_\infty,G), \{r=0\})\geq G^2/2.
\end{equation}
so $\mathcal{G}$ is a closed set. Moreover, since the Hamiltonian \eqref{eq:oneparameterHamiltonian} depends analytically on $G$, and, by \eqref{eq:distancehomocliniczero}, for all  $G\in\mathcal{G}$ the vector field associated with \eqref{eq:oneparameterHamiltonian} is analytic on a neighbourhood of $W^\pm(\gamma_\infty,G)$, there exists an open subset $\mathcal{O}\subset\mathcal{\widetilde{G}}$ such that  $\mathcal{G}\subset\mathcal{O}$ and in which $\varphi^-_G\in C^\omega([R_1,R_2]\times\mathcal{O})$. Define now the function $\Delta (r,G):[R_1,R_2]\times\mathcal{\widetilde{G}}\to\mathbb{R}$ given by
\[
\Delta (r,G)=\varphi^-_G(r)-\partial_r S^+(r,0;G)
\]
which satisfies $\varDelta(r,G)=0$ for all $G\in \mathcal{\widetilde{G}}$ and $\varDelta\in C^\omega\left([R_1,R_2]\times\mathcal{O}\right)$. Suppose now that $\mathcal{G}=\{ \delta\leq |G|\leq 2G_*\}$. Then, $\Delta(r,G)=0$ for all $(r,G)\in[R_1,R_2]\times\{ \delta\leq |G|\leq 2G_*\}$ and we obtain a contradiction with the fact that for $|G|\geq  G_*$ the manifolds $W^{u}(\gamma_\infty,G)$, $W^{s}(\gamma_\infty,G)$ intersect transversally (see Theorem \ref{thm:transversalityjde}). Therefore, $\mathcal{G}\subsetneq\{ \delta\leq |G|\leq 2G_*\}$. We now show that, moreover, $\mathcal{G}$ cannot contain any accumulation point. To see this suppose that there exists $G_{max}\in \mathcal{G}$ such that 
\[
G_{max}=\max\{G\in \mathcal{G}\colon G\ \text{is an accumulation point of}\ \mathcal{G}\}.
\]
Since $G_{max}\in\mathcal{G}$ there exists an open interval $\mathcal{V}\subset\mathcal{O}$ such that $G\in \mathcal{V}$. Then, the fact that $\mathcal{V}\subset\mathcal{O}$ implies that$ \varDelta(r,G)\in C^\omega\left([R_1,R_2]\times\mathcal{V}\right)$ and since $G_{max}$ is an acuumulation point of $\mathcal{G}$ we conclude that $\Delta(r,G)=0$ on $[R_1,R_2]\times\mathcal{V}$. Then $\mathcal{V}\subset \mathcal{G}$, so there exists $\tilde{G}\in\mathcal{V}\subset\mathcal{G}$ such that $\tilde{G}>G_{max}$, a contradiction with the definition of $G_{max}$. 
\end{proof}\vspace{0.3cm}

Denote now by $\mathcal{J}\subset\mathbb{R}$ the set 
\begin{equation}\label{eq:goodsetG}
\mathcal{J}=\{ G\in J\colon G\neq 0,\  W^+(\gamma_\infty;G)\neq W^-(\gamma_\infty;G)\}
\end{equation}
where $J$  was defined in Lemma \ref{lem:monotonicitylemma} (see also Theorem \ref{thm:convergencePS}).

\begin{lem}\label{lem:isolatedcriticalpoints}
For all $G\in \mathcal{J}$ the set $\mathrm{Crit}(\mathcal{A}_G)=\{\varphi\in D^{1,2}\colon \mathrm{d} \mathcal{A}_G(\varphi)=0\}$ is isolated in $D^{1,2}$.
\end{lem}

\begin{proof}
Following \cite{MR1487629}, we define the map $T_R :\mathrm{Crit}(\mathcal{A})\subset D^{1,2}\to \mathbb{R}$ given by 
\[
T_R=\sup \{s\in \mathbb{R}\colon r_0(s)+\varphi(s)=R,\ \varphi\in \mathrm{Crit}(\mathcal{A})\subset D^{1,2}\}.
\]
We now show that the set $T_R (\mathrm{Crit}(\mathcal{A}))$ is isolated in $\mathbb{R}$. Suppose on the contrary that there exists an accumulation point $T_*\in T_R(\mathrm{Crit}(\mathcal{A}))$, then, there exist $\{(\varphi_n,t_n)\}_{n\in\mathbb{N}}\subset\mathrm{Crit}(\mathcal{A})\times \mathbb{R}$ and $R\in\mathbb{R}_+$ such that  $t_n\to T_R$,  $(r_0+\varphi_n)(t_n)=R$ and
\[
((r_0+\varphi_n)(t_n), (\dot{r}_0+\dot{\varphi}_n)(t_n))\in W^{+}_{\mathrm{loc}}(\gamma_\infty;G).
\]
Thus,  there exist infinitely many different homoclinic points contained in the piece of the local stable manifold $\gamma_+=\{y=\partial_r S^+(r,t),\ t=T_*,\ r\in [R,R_1]\}$ for any $R_1<R$. This would imply the existence of $T_{**}<T_*$,  $R_2<R_3$ such that $\gamma_+\cap \phi^{T_*-T_{**}} (\gamma_-)$ intersect at infinitely many points, where $\gamma_-=\{y=\partial_r S^-(R,t),\ t=T_{**},\ r\in [R_2,R_3]\}$.However, $\gamma_+$ and $\gamma_-$ are compact analytic curves, and since $G\in\mathcal{J}$ they cannot intersect at infinitely many points.

 By Lemma 3.3.  in \cite{MR1487629}, the function $T_R :\mathrm{Crit}(\mathcal{A})\subset D^{1,2}\to \mathbb{R}$ is continuous, so the lemma is proven, for if it were to be false there would exist an accumulation point $T_*\in T_R(\mathrm{Crit}(\mathcal{A}))$.
\end{proof}

The fact that the critical points are isolated implies the following non-degeneracy property at, at least, one of the critical points of $\mathcal{A}_G$ at the level $c_G$. We say that $\varphi_*\in \mathrm{Crit}(\mathcal{A}_G)$ has a \emph{local mountain pass structure} if, for all neighbourhood $U\subset  D^{1,2}$ of $\varphi_*$, the set $\{\varphi\in U\colon \mathcal{A}_G(\varphi)<\mathcal{A}_G(\varphi_*) \}$ is not path connected. The following result is a direct consequence of Lemma \ref{lem:isolatedcriticalpoints}  and Theorem 1 in \cite{MR843584}.

\begin{prop}\label{prop:localmpstructure}
For all $G\in\mathcal{J}$ there exists $\varphi_*\in \mathrm{Crit}(\mathcal{A}_G)$ such that $\mathcal{A}_G(\varphi_*)=c_G$ which has a local mountain pass structure.
\end{prop}

\begin{rem}
In all the forthcoming sections we fix $G\in\mathcal{J}$ where $\mathcal{J}$ is the set defined in \eqref{eq:goodsetG} and omit the dependence on $G$.
\end{rem}

\subsection{The reduced action functional}

For $n\in \mathbb{N}\setminus \{0\}$ we denote by $H^{1}([-n,n])$ the usual Sobolev space consisting of functions defined on the interval $[-n,n]\subset \mathbb{R}$ with one weak derivative in $L^2([-n,n])$ and introduce the restriction operator
\begin{equation}\label{eq:restricition}
\begin{aligned}
j:D^{1,2}& \longrightarrow  H^1([-n,n])\\
\varphi & \longmapsto j(\varphi)=\varphi|_{[-n,n]}
\end{aligned}
\end{equation}
Then, for a sufficiently small neighbourhood $\tilde{U}\subset H^1([-n,n])$ of a point $\tilde{\varphi}_*=j(\varphi_*)$ where $\varphi_*\in D^{1,2}$  and a sufficiently large $n\in\mathbb{N}$ (depending on $\varphi_*$) we define the reduced action functional $\widetilde{\mathcal{A}}: \tilde{U}\subset H^1([-n,n]) \rightarrow \mathbb{R}$  given by 
\[
\widetilde{\mathcal{A}}(\tilde{\varphi})= \int_{-n}^{n} \mathcal{L}_{\mathrm{ren}} (\tilde{\varphi},\dot{\tilde{\varphi}},s)\mathrm{d}s - S^+ ((r_0+\tilde{\varphi})(n))+S^-((r_0+\tilde{\varphi})(-n))+\dot{r}_0(n)(\tilde{\varphi}(n)-\tilde{\varphi}(-n)),
\]
where the renormalized Lagrangian $\mathcal{L}_{\mathrm{ren}}$ is defined in \eqref{eq:renormalizedLagrangian} and $S^\pm$ are the generating functions of the local stable and unstable manifolds which were obtained in Proposition \ref{prop:generatingfunctions}. Notice that for $n$ sufficiently large (depending on $\varphi_*$) and $\tilde{\varphi}$ sufficiently close to $j(\varphi_*)$ the values $(r_0+\tilde{\varphi})(\pm n)$ are contained in $\mathrm{Dom}( S^\pm)$.

We now want to translate the results we have obtained for the functional $\mathcal{A}$, in particular Proposition \ref{prop:localmpstructure},  in results for the functional $\mathcal{\widetilde{A}}$. To that end, given any constant $c\in\mathbb{R}$ and $n\in\mathbb{N}$  we introduce the functional spaces
\[
\begin{split}
D^{1,2}_+(c,n)=\{\varphi\in C([n,\infty))\colon &  \exists v_\varphi \in L^2([n,\infty))\ \text{such that } \\ &\varphi(s)=c+\int_{n}^s v_\varphi(t)\mathrm{d}t,\ \forall s\in [n,\infty)\}
\end{split}
\]
and
\[
\begin{split}
D^{1,2}_-(c,n)=\{\varphi\in C((-\infty,-n])\colon & \exists v_\varphi \in L^2((-\infty,-n])\ \text{such that } \\ &\varphi(s)=c-\int_s^{-n} v_\varphi(t)\mathrm{d}t,\ \forall s\in (-\infty,-n] \}
\end{split}
\]
Define also the weakly closed subsets 
\[
\begin{split}
\tilde{D}^{1,2}_+ (c,n)=&\{\varphi\in D^{1,2}_+(c,n)\colon r_0(s)+\varphi(s)\geq r_0(n)+c,\ \forall s\in[n,\infty)\} \\
\tilde{D}^{1,2}_- (c,n)=&\{\varphi\in D^{1,2}_-(c,n)\colon r_0(s)+\varphi(s)\geq r_0(-n)+c,\ \forall s\in(-\infty,-n]\}\\
\end{split}
\]
Then, we define the asymptotic actions 
\begin{equation}\label{eq:asymptoticactions}
\mathcal{A}^\pm(\varphi)=\pm \int_{\pm n}^{\pm\infty} \mathcal{L}_{\mathrm{ren}}(\varphi,\dot{\varphi},s)\mathrm{d}s
\end{equation}

\begin{lem}\label{lem:asymptoticactions}
For all $c\in\mathbb{R}$ there exists $n_0\in\mathbb{N}$  such that for all $n\geq n_0$ there exists a unique $\varphi_\pm\in \tilde{D}^{1,2}_\pm(c,n)$ such that 
\[
\mathcal{A}^\pm(\varphi_\pm)=\min\{ \mathcal{A}^\pm(\psi) \colon \psi\in \tilde{D}^{1,2}_\pm(c,n)) \}.
\]
Moreover,
\[
 \mathcal{A}^\pm (\varphi_\pm)=\mp S^\pm ((r_0(\pm n)+c)\pm S^0((r_0)(\pm n))\mp \dot{r}_0(\pm n)c.
\]

\end{lem}

\begin{proof}
A simple computation shows that 
\[
\partial_{rr}^2 V(r,t)=-\frac{3G^2}{r^4}+\frac{3r^2}{(r^2+\rho^2(t))^{5/2}}-\frac{1}{(r^2+\rho^2(t))^{3/2}}
\]
from where we easily deduce that there exists $R>0$ such that, if 
\[
r_0(n)+c\geq R
\]
then, the functional $\varphi\mapsto \mathcal{A}^+(\varphi)$ in \eqref{eq:asymptoticactions} is strictly convex on the strictly convex set $\tilde{D}^{1,2}_+(c,n)$. Therefore, there exists a unique minimizer $\varphi_+\in \tilde{D}^{1,2}_+(c,n)$  for which 
\[
\mathcal{A}^+(\varphi_+)=\min\{ \mathcal{A}^+(\psi) \colon \psi\in \tilde{D}^{1,2}_+(\tilde{\varphi}(n)) \}.
\]
Moreover, is easy to check that $\varphi_+$ is a critical point of the functional $\mathcal{A}^+(\varphi)$. Consequently, $r(s)=r_0(s)+\varphi_+(s)$ is an orbit of \eqref{eq:Hamiltonian} asymptotic in the future to $\gamma_\infty$. 

Let now $S^+(r,s)$ be the generating function of the local stable manifold introduced in Proposition \ref{prop:generatingfunctions}. By uniqueness of the local stable manifold, the function $\varphi_+(s)$ satisfies that 
\[
(\dot{r}_0+\dot{\varphi}_+)(s)= \partial_r S^+ (r_0(s)+\varphi_+(s),s)
\]
for all $s\in [n,\infty)$. In particular, since moreover $\varphi_+(s)\in D^{1,2}_+$ Lemma \ref{lem:differencegeneratingfunctions} in Appendix \ref{sec:technicallemmas} implies that $|\dot{\varphi}_+(s)|\lesssim s^{1/6}$ as $s\to\infty$ and since $\dot{r}_0(s)\sim s^{-1/3}$ as $s\to\infty$ ,  we can integrate by parts to obtain
\[
\begin{split}
\int_n^\infty \mathcal{L}_{\mathrm{ren}}(\varphi_+,\dot{\varphi}_+,s)= & \int_n^\infty \frac{\dot{\varphi}_+^2}{2}+V(r_0+\varphi_+)-V_0(r_0)-\ddot{r}_0\varphi_+\\
=& -\dot{r}_0(n) c+ \int_n^\infty \frac{\dot{\varphi}_+^2}{2}+\dot{r}_0\dot{\varphi}+V(r_0+\varphi_+)-V_0(r_0).
\end{split}
\]
On the other hand,
\[
\begin{split}
\int_n^\infty \frac{\dot{\varphi}_+^2}{2}+\dot{r}_0\dot{\varphi}+V(r_0+\varphi_+)-&V_0(r_0)\\
=&  \int_n^\infty \left( (\dot{r}_0+\dot{\varphi}_+)\partial_r S^+ (r_0+\varphi_+)-H(r_0+\varphi_+,\partial_r S^+ (r_0+\varphi_+),s) \right.\\
& \left. -  \dot{r}_0\  \partial_r S^0 (r_0)- H_0(r_0,\partial_r S^0(r_0)) \right)\\
=&\int_n^\infty \frac{\mathrm{d}}{\mathrm{d}s} S^+ ((r_0+\varphi_+)(s))-\frac{\mathrm{d}}{\mathrm{d}s} S^0 (r_0(s))\\
=&- S^+((r_0(n)+c)+S^0(r_0(n))
\end{split}
\]
where we have used that $ H(r_0+\varphi_+,\partial_r S^+ (r_0+\varphi_+),s)+\partial_t S^+(r_0+\varphi_+,s)=0$ and the fact that
\[
\lim_{s\to\infty} S^+( (r_0+\varphi_+)(s))-S^0(r_0(s))=0,
\]
which is also proved in Lemma \ref{lem:differencegeneratingfunctions}.
\end{proof}

Introduce  now the extension operator $E: \tilde{U}\subset H^1([-n,n]) \to D^{1,2}$ 
\begin{equation}\label{eq:extensionoperator}
E(\tilde{\varphi})=\left\{ \begin{array}{ccr} E_- (\tilde{\varphi}) &\text{for}& \quad s\in(-\infty,-n) \\
 \tilde{\varphi}&\text{for}&\quad s\in [-n,n]\\
E_+ (\tilde{\varphi}) &\text{for}&\quad s\in(n,\infty)  \\ \end{array} \right. 
\end{equation}
where 
\[
\begin{split}
E_\pm(\tilde{\varphi})= &\{\varphi \in \tilde{D}^{1,2}_\pm(\tilde{\varphi}(\pm n))\colon \mathcal{A}^\pm (\varphi)\leq \mathcal{A}^\pm(\psi),\ \forall\psi\in \tilde{D}^{1,2}_\pm(\tilde{\varphi}(\pm n))\}.\\
\end{split}
\]

From the proof of Lemma \ref{eq:asymptoticactions} we deduce the following.

\begin{lem}\label{lem:welldefinedextension}
Let  $\varphi\in D^{1,2}$, let $n\in\mathbb{N}$ sufficiently large, let $\tilde{\varphi}=j(\varphi)$  and let $\widetilde{U}\subset H^1([-n,n])$ a sufficiently small neighbourhood of $\tilde{\varphi}$. Then, the extension operator \eqref{eq:extensionoperator} is well defined on $\widetilde{U}$.
\end{lem}

The proof of the following Lemma is an straightforward consequence of the definition of the extension operator $E$. 

\begin{lem}\label{lem:relationshihpfunctionals}
Let $\varphi_*\in D^{1,2}$. Then, for $n\in\mathbb{N}$ sufficiently large and all  $\varphi$ contained in a sufficiently small neigubourhood $U\subset D^{1,2}$ of $\varphi_*$
\[
\mathcal{\widetilde{A}}(j(\varphi))\leq \mathcal{A}(\varphi).
\]
Also, for all $\tilde{\varphi}$ in a sufficiently small neighbourhood $\tilde{U}\subset H^1([-n,n])$ of $j(\varphi_*)$ 
\[
\mathcal{\widetilde{A}}(\tilde{\varphi})= \mathcal{A}(E(\tilde{\varphi})).
\]
Moreover, for $\varphi_*\in D^{1,2}$ such that $\mathrm{d}\mathcal{A}(\varphi_*)=0$ we have  $\mathrm{d}\mathcal{ \widetilde{A}}(j(\varphi_*))=0$.
\end{lem}

We can now translate the result for $\mathcal{A}$ stated in Proposition \ref{lem:isolatedcriticalpoints} in an analogous result for $\mathcal{\tilde{A}}$. 

\begin{prop}\label{prop:localmountainpassrestricted}
There exists $n\in\mathbb{N}$ and  $\tilde{\varphi}_*\subset H^{1}([-n,n])$  which is a critical point of $\widetilde{\mathcal{A}}$ and has a local mountain pass structure. 
\end{prop}

\begin{proof}
The proof is a simple combination of the proof of Theorem 1 in \cite{MR843584} together with the relationship between the functionals $\mathcal{A}$ and $\tilde{\mathcal{A}}$ which was obtained in Lemma \ref{lem:relationshihpfunctionals}. We sketch here the details for the sake of completeness.

Denote by $\mathrm{Crit}(\mathcal{A},c_G)=\{\varphi\in \mathrm{Crit}(\mathcal{A})\subset D^{1,2}\colon \mathcal{A}(\varphi)=c_G \}$ where $c_G$ is the critical value defined in \eqref{eq:definitioncriticalvalue}. Lemma \ref{lem:isolatedcriticalpoints} implies, in particular, that $\mathrm{Crit}(\mathcal{A},c_G)$ is an isolated subset in $D^{1,2}$. Moreover, fixing $\overline{m}$ sufficiently large $\mathrm{Crit}(\mathcal{A},c_G)\subset \mathcal{F}_{\overline{m}}$ where $\mathcal{F}_{\overline{m}}$ was defined in \eqref{eq:constrainedset}. Let now $\varepsilon>0$ and $\gamma_\varepsilon\subset \varGamma\subset D^{1,2}$ be a suboptimal path at level $c_G$. Then, $\gamma_\varepsilon$ intersects a finite number of elements in  $\mathrm{Crit}(\mathcal{A},c_G)$, which we denote by $\{\varphi_1,\dots,\varphi_k\}$ for some finite $k$. Let now $\delta>0$ sufficiently small and denote by $\mathcal{B}_{i,\delta}\subset D^{1,2}$ the ball of radius $\delta$ around $\varphi_i$. Without loss of generality we can assume that $\gamma_\varepsilon$ intersects each $\mathcal{B}_{i,\delta}$ only once so we can define (see Figure \ref{fig:localmountainpass})
\[
t_i^-=\inf\{t\in[0,1]\colon \gamma(t)\in \mathcal{B}_{i,\delta} \}\qquad\qquad t_i^+=\sup\{t\in[0,1]\colon \gamma(t)\in \mathcal{B}_{i,\delta} \}.
\]

\begin{figure}\label{fig:localmountainpass}
\centering
\includegraphics[scale=0.60]{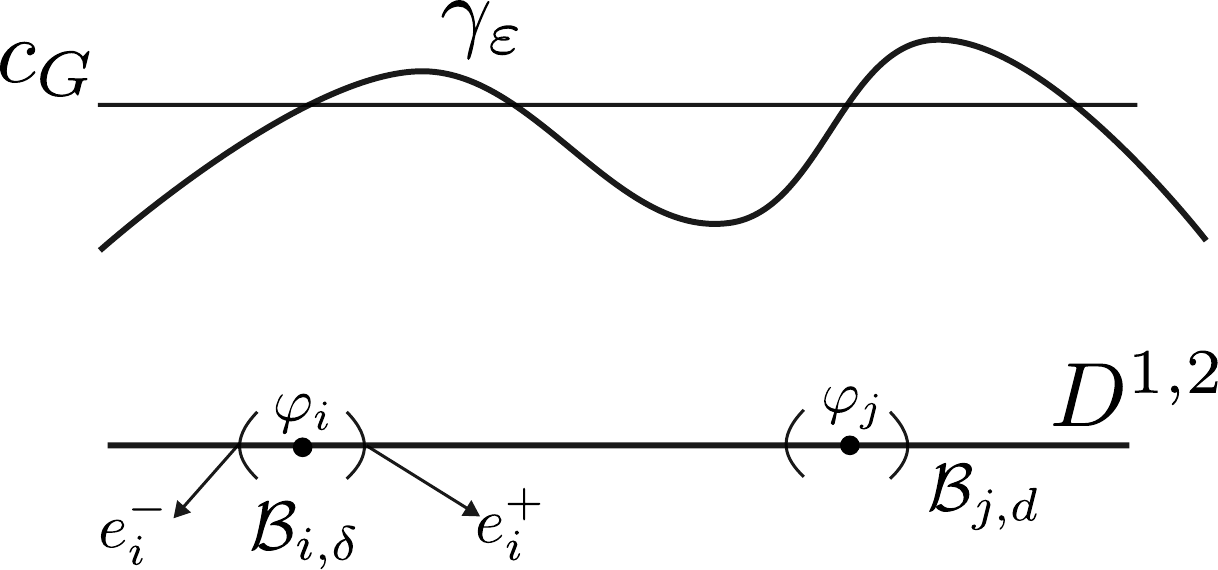}
\caption{Sketch of the suboptimal path $\gamma_\varepsilon$.}
\end{figure}

\noindent and  $e^-_i=\gamma(t_i^{-})$ and  $e^+_i=\gamma(t_i^{+})$.  Let now $n\in\mathbb{N}$ large enough and $\delta$ small enough so the restricition operator $j$ in\eqref{eq:restricition} is well defined on $\cup_{1\leq i\leq k} \mathcal{B}_{i,\delta}$. Let $\tilde{\mathcal{B}}_{i,\delta}=j(\mathcal{B}_{i,\delta})$. Again, without loss of generality, we can assume that for all $i=1,\dots,k$,   $e_i^\pm\in D^{1,2}$ have minimizing tails, that is, $e_i^\pm|_{[n,\infty)}\subset \tilde{D}^{1,2}(e_i^\pm(n),n)$ is the unique minimizer of $\mathcal{A}^+$ on $\tilde{D}^{1,2}(e_i^\pm(n),n)$ and $e_i^\pm|_{(-\infty,-n]}\subset \tilde{D}^{1,2}(e_i^\pm(-n),n)$ is the unique minimizer of $\mathcal{A}^-$ on $\tilde{D}^{1,2}(e_i^\pm(-n),n)$. Now, define the paths
\[
\tilde{\gamma}_i=j(\gamma|_{[t_i^-t_i^+]})\subset H^{1}([-n,n])
\]
for $i=1,\dots,k$, and the points $\tilde{\varphi}_i=j(\varphi_i)\in H^{1}([-n,n])$, which are indeed critical points of the reduced action functional $\mathcal{\tilde{A}}$. Suppose the point $\tilde{\varphi}_i$ does not have a local mountain pass structure. Then, we can build (see Lemma 1 in \cite{MR843584}) a continuous deformation $\eta:[0,1]\times \tilde{\mathcal{B}}_{i,\delta}\to H^{1}([-n,n])$ such that 
\[
\eta\left(\{1\}\times (\{\tilde{\varphi}\in  \tilde{\mathcal{B}}_{i,\delta}\colon \mathcal{\tilde{A}}(\varphi)\leq c_G+\varepsilon\}\setminus \tilde{\mathcal{B}}_{i,\delta/2})\right)\subset \{\tilde{\varphi}\in  \tilde{\mathcal{B}}_{i,\delta}\colon \mathcal{\tilde{A}}(\varphi)\leq c_G-\varepsilon\}
\]
\[
\eta \left([0,1]\times \mathrm{Cl}(\tilde{\mathcal{B}}_{i,\delta/2})\right)\subset \tilde{\mathcal{B}}_{i,\delta}
\]
\[
\eta(z,\varphi)=\varphi\qquad \forall (z,\varphi)\in [0,1]\times \{\tilde{\varphi}\in  \tilde{\mathcal{B}}_{i,\delta}\colon |\mathcal{\tilde{A}}(\varphi)-c_G|\geq\varepsilon\}\ .
\]
where by $\mathrm{Cl}(\tilde{\mathcal{B}}_{i,\delta/2})$ we denote the closure of the ball $\tilde{\mathcal{B}}_{i,\delta/2}$. Write $\eta(\tilde{\gamma}_i)=\eta(\{1\}\times \tilde{\gamma}_i)$, which satisfies 
\[
\mathcal{\tilde{A}}(\eta(\tilde{\gamma}_i))\leq c_G-\varepsilon
\]
and
\[
\eta(\tilde{\gamma}_i)(t_i^-)=j(e_i^-)\qquad\qquad \eta(\tilde{\gamma}_i)(t_i^+)=j(e_i^+).
\]
Now, for the extension operator $E$ is well defined on $\tilde{\mathcal{B}}_{i,\delta}$  (shrinking $\delta$ if necessary) and $\eta(\tilde{\gamma}_i)\subset \tilde{\mathcal{B}}_{i,\delta}$,  we can define the path $E(\eta(\tilde{\gamma}_i))\subset D^{1,2}$ which, by construction, satisfies
\[
\mathcal{A} (E(\eta(\tilde{\gamma}_i)))\leq c_G-\varepsilon
\]
and 
\[
\eta(\tilde{\gamma}_i)(t_i^-)=e_i^-\qquad\qquad \eta(\tilde{\gamma}_i)(t_i^+)=e_i^+.
\]
The proposition is therefore proved for,  if none of the points $\tilde{\varphi}_i$ posses a local mountain pass structure, the continuous path $\gamma\subset D^{1,2}$ defined by gluing (in the obvious way) the segments $\gamma_\varepsilon\setminus \bigcup_{1\leq i\leq k} \gamma_\varepsilon|_{[t_i^-,t_i^+]}$ with the segments $E(\eta_i(\tilde{\gamma}_i))$ satisfies $\mathcal{A}(\gamma)\leq c-\varepsilon$, a contradiction.
\end{proof}

Proposition \ref{prop:localmountainpassrestricted} entails a non degeneracy condition for the intersection of the invariant manifolds $W^+(\gamma_\infty)$ and   $W^-(\gamma_\infty)$ at the homoclinic orbit associated with $\tilde{\varphi}_*$. We now make use of topological degree theory to exploit this non degeneracy condition.  Let $\tilde{\varphi}_*\in H^1([-n,n])$ be the critical point obtained in Proposition \ref{prop:localmountainpassrestricted} and consider a sufficiently small neighbourhood $\widetilde{U}\in H^1([-n,n])$ such that $\tilde{\varphi}_*\in \widetilde{U}$. By definition of the functional $\widetilde{\mathcal{A}}$, and the fact that 
\[
\min_{s\in[-n,n]} r_0(s)+\tilde{\varphi}_*(s)>0
\]
the differential $\mathrm{d}\widetilde{\mathcal{A}}(\tilde{\varphi}):\widetilde{U}\to H^1([-n,n])$ is a continuous linear functional and, for any $\tilde{\varphi}\in\widetilde{U}$ and any $\psi\in H^1([-n,n])$, we can express  
\begin{equation}\label{eq:differentialreducedaction}
\mathrm{d} \mathcal{\widetilde{A}}(\varphi)[\psi]=\langle \dot{\varphi},\dot {\psi} \rangle_{L^2([-n,n])} + 2 \int_{-n}^n \frac{\varphi\psi}{r_0^3}+ P(\varphi)[\psi],
\end{equation}
where we have introduced the functional (compare expression \eqref{eq:identitypluscompact} in the proof of  Lemma \ref{lem:propermap})
\[
\begin{split}
\widetilde{P}(\varphi)[\psi] =& \int_{-n}^n\left( \frac{G^2}{(r_0+\varphi)^3}-\frac{G_0}{r_0^3} \right)\psi -\int_{-n}^n\left( \frac{r_0+\varphi}{((r_0+\varphi)^2+\rho^2)^{3/2}}-\frac{1}{r_0^2} +\frac{2\varphi}{r_0^3}\right)\psi\\
-& \left(\partial_r S^+((r_0+\tilde{\varphi})(n))-\dot{r}_0(n)\right) \psi(n)+\left(\partial_r S^-((r_0+\tilde{\varphi})(-n))-\dot{r}_0(-n)\right)\psi(-n).
\end{split}
\]
Since $r_0(s)>0$ and the interval $[-n,n]$ is compact, the expression
\[
\langle\langle \varphi,\psi \rangle\rangle=\langle \dot{\varphi},\dot {\psi} \rangle_{L^2([-n,n])} + 2 \int_{-n}^n \frac{\varphi\psi}{r_0^3},
\]
defines an equivalent inner product on $H^1([-n,n])$. For $\tilde{\varphi}\in \widetilde{U}$, denote by $\nabla\widetilde{\mathcal{A}}(\tilde{\varphi})$ the unique element of $H^1([-n,n])$ such that  for all $\psi\in H^1([-n,n])$
\begin{equation}\label{eq:finalequivalentinnerproduct}
\langle\langle \nabla\widetilde{\mathcal{A}}(\tilde{\varphi}),\psi\rangle\rangle=\mathrm{d} \mathcal{\widetilde{A}}(\varphi)[\psi].
\end{equation}
From \eqref{eq:differentialreducedaction} one easily deduces that the map $\nabla\widetilde{\mathcal{A}}:\widetilde{U}\to H^1([-n,n])$ is a compact perturbation of the identity. Therefore, for any subset $\widetilde{V}\subset \widetilde{U}\in H^1([-n,n])$ and any point $\tilde{z}\in  H^1([-n,n])$ such that $\tilde{z}\notin \nabla\widetilde{\mathcal{A}}(\partial \widetilde{V})$ the Leray-Schauder degree \footnote{The Leray-Schauder degree is a generalization of the Brouwer degree to maps between infinite dimensional spaces wich are of the form identity+compact. Details about its definition and properties can be found in \cite {MR2859263}.} associated with the triple $(\nabla \widetilde{\mathcal{A}},\widetilde{V},\tilde{z})$, which we denote by 
\[
\mathrm{deg}(\nabla \widetilde{\mathcal{A}},\widetilde{V},\tilde{z}),
\]
is well defined. Proposition \ref{prop:localmountainpassrestricted}, together with Theorem 2 in \cite{MR843584}, imply the following result.
\begin{prop}\label{prop:degreeproposition}
Let $\tilde{\varphi}_*\in H^{1}([-n,n])$  be the critical point of $\widetilde{\mathcal{A}}$ which was obtained in Proposition \ref{prop:localmountainpassrestricted} and, for $\varepsilon>0$, denote by $B_{\varepsilon}(\tilde{\varphi}_*)\subset H^{1}([-n,n])$ the ball of radius $\varepsilon$ centered at  $\tilde{\varphi}_*$. Then, there exists $\varepsilon_0$ such that for all $0\leq \varepsilon\leq\varepsilon_0$,  
\[
\mathrm{deg}(\nabla{\widetilde{A}},B_\varepsilon(\tilde{\varphi}_*),0)=-1.
\]
\end{prop}

As a consequence of Proposition \ref{prop:degreeproposition} we can prove that the manifolds $W^+(\gamma_\infty;G)$ and $W^-(\gamma_\infty;G)$ intersect transversally for $G\in\mathcal{G}$. First, we introduce some notation which will be useful in the proof of Proposition \ref{prop:topologicaltransversality} and in Section \ref{sec:multibumpsolutions}. Let  $s\in [-n,n]$, then, we  denote by $\mathrm{ev}_s: H^{1}([-n,n])\to\mathbb{R}$ the evaluation operator given by 
\[
\mathrm{ev}_s \tilde{\varphi}=\tilde{\varphi}(s).
\]
In addition we denote by $\mathtt{ev}_s\in H^1([-n,n])$ the  unique element such that for all $\psi\in H^1([-n,n])$
\[
\langle \langle \mathtt{ev}_s,\psi \rangle\rangle=\mathrm{ev}_s (\psi).
\]

\begin{prop}\label{prop:topologicaltransversality}
For all $G\in\mathcal{J}$ there exists a topologically transverse intersection between $W^{+}(\gamma_\infty;G)$ and $W^{-}(\gamma_\infty;G)$.
\end{prop}

\begin{proof}
Let $\tilde{\varphi}_*\subset H^1([-n,n])$ be the critical point obtained in Proposition \ref{prop:localmountainpassrestricted}. Then, there exists $\varepsilon_0>0$ such that for all $0\leq\varepsilon\leq\varepsilon_0 $
\[
\nabla{\tilde{\mathcal{A}}}(\tilde{\varphi}_*)=0\qquad\qquad\text{and}\qquad\qquad \nabla{\tilde{\mathcal{A}}}(\tilde{\varphi})\neq0\qquad\forall \varphi\in B_\varepsilon(\varphi_*)\setminus\{\varphi_*\}.
\]
In particular, there exists $\delta_0>0$ such that 
\[
\sup_{\varphi\in \partial B_\varepsilon(\varphi_*)}\lVert \nabla{\tilde{\mathcal{A}}}(\tilde{\varphi}) \rVert\geq \delta_0.
\]
Define now, for $\delta\in\mathbb{R}$, the one parameter family of maps $F_\delta: H^1([-n,n])\to  H^1([-n,n])$ given by 
\[
F_\delta(\varphi)=\nabla{\tilde{\mathcal{A}}}(\tilde{\varphi}_*)+\delta \mathtt{ev}_n= \nabla \left(\int_{-n}^n \mathcal{L}_{\mathrm{ren}}(\varphi,\dot{\varphi},s)\right)+\partial_r S^-((r_0+\varphi)(-n))\mathtt{ev}_{-n}-(\partial_r S^+((r_0+\varphi)(n))+\delta)\mathtt{ev}_{n}
\]
Then, it is possible to find $\delta_1>0$ such that $F_\delta(\varphi)$ is an admissible homotopy for $\delta\in[- \delta_1,\delta_1]$ so by invariance of the degree under admissible homotopies
\[
\mathrm{deg}(F_\delta,B_\varepsilon(\varphi_*),0)=-1\qquad\qquad \forall\delta\in[-\delta_1,\delta_1].
\]
We now show how this implies the desired conclusion. Let $Q=\{\varphi\in B_\varepsilon\colon F_\delta(\varphi)=0,\ \delta\in[-\delta_1,\delta_1]\}$. Then, denoting by $\pi_r,\pi_y$ the projections onto the $r,y$ coordinates of a point $(r,y,t)\in \mathbb{R}^2\times\mathbb{T}$, and by $\phi^s$ the flow at time $s$ associated to Hamiltonian \eqref{eq:oneparameterHamiltonian} we have that 
\[
[-\delta_1,\delta_1]\subset \{ \pi_y\circ\phi^{2n}_H (r,\partial_r S^-(r,-n),-n)-\partial_r S^+(\pi_r \circ \phi^{2n}_H (r,\partial_r S^-(r,-n),-n)\colon r\in R_{\delta_1}\}
\]
for $R_{\delta_1}\{r=(r_0+\varphi)(-n)\colon \varphi\in Q\}$. This completes the proof.
\end{proof}


\section{Construction of multibump solutions}\label{sec:multibumpsolutions}

We now show how Proposition \ref{prop:degreeproposition} together with the parabolic lambda Lemma \ref{lem:lambdalemma} can be used to the deduce the existence of homoclinic orbits to $\gamma_\infty$ which perform any arbitrary number of ``bumps". We start by stating the following lemma, which is nothing but a reformulation of the parabolic lambda Lemma \ref{lem:lambdalemmaoriginalcoord}.

\begin{lem}\label{lem:usefulparaboliclemma}
There exists $R$ large enough such that for $R_0,R_1\geq R$ there exists $T_*$ such that for all $T\geq T_*$ there exists a unique orbit $\hat{r}(t; T,R_0,R_1)$ of \eqref{eq:Hamiltonian} for which $\hat{r}(0)=R_0$ and $\hat{r}(T)=R_1$. Moreover, for all $\varepsilon>0$ there exists $T_{**}(\varepsilon)$ such that for all $T\geq T_{**}$  the unique solution $\hat{r}(t; T,R_0,R_1)$ satisfies
\[
\partial_r S^+ (R_0)-\dot{\hat{r}}(0)\leq \varepsilon\qquad\qquad  \dot{\hat{r}}(T)-\partial_r S^- (R_1) \leq \varepsilon.
\]
\end{lem}
Given $R_0,R_1\geq R$ and $T\geq T_*$ we denote by 
\[
v^+(T,R_0,R_1)=\dot{\hat{r}}(0;T,R_0,R_1)\qquad\qquad v^-(T,R_0,R_1)=\dot{\hat{r}}(T;T,R_0,R_1).
\]
where $\hat{r}(t;T,R_0,R_1)$ is the orbit segment found in Lemma \ref{lem:usefulparaboliclemma}.

\subsection{Proof of Theorem  \ref{thm:maintheorem}}

We are now ready to build the multibump solutions. By proposition \ref{prop:degreeproposition} we know that there exists a critical point $\tilde{\varphi}_*\in H^{1}([-n,n])$  of $\widetilde{\mathcal{A}}$ and $\varepsilon_0\geq 0$ such that  for all $0\leq \varepsilon\leq\varepsilon_0$,  
\[
\mathrm{deg}(\nabla{\widetilde{A}},B_\varepsilon(\tilde{\varphi}_*),0)=-1,
\]
where $B_{\varepsilon}(\tilde{\varphi}_*)\subset H^{1}([-n,n])$ stands for the ball of radius $\varepsilon$ centered at  $\tilde{\varphi}_*$. For any $L\in\mathbb{N}$, introduce now the map 
\begin{equation}\label{eq:multibumpmap}
\begin{aligned}
F: (B_\varepsilon(\varphi_*))^{L+1} \times (\{l\in\mathbb{N}\colon l\geq T_{**}\})^{L}& \longrightarrow (H^1([-n,n]))^{L+1}\\
(\varphi_1,\dots,\varphi_{L+1},l_1,\dots,l_L)& \longmapsto (F_1,\dots,F_{L+1})
\end{aligned}
\end{equation}
where the maps $F_j$, $1\leq j\leq L+1$ are given by 
\[
\begin{split}
F_1=&\nabla\mathcal{\widetilde{A}}_G  +\left(\partial_r S^+(\varphi_1(n),n)-v^+(l_1,\varphi_1(n),\varphi_2(-n))\right)\mathtt{ev}_n\\
F_{L+1}=&\nabla \mathcal{\widetilde{A}}_G+ \left(v^-(l_{L},\varphi_{L}(n),\varphi_{L+1}(-n))- \partial_r S^-(\varphi_{L+1}(-n),-n)\right)\mathtt{ev}_{n}\\
\end{split}
\]
and for $2\leq j\leq L$ (this set is empty for $L=1$)
\[
\begin{split}
F_j=&\nabla\mathcal{\widetilde{A}}_G+ \left(\partial_r S^+(\varphi_j(n),n)-v^+(l_{j},\varphi_{j}(n),\varphi_{j+1}(-n))\right)\mathtt{ev}_{n}\\
&+\left( v^-(l_{j-1},\varphi_{j-1}(n),\varphi_j(-n))-\partial_r S^-((\varphi_j)(-n),-n)\right)\mathtt{ev}_{-n}.
\end{split}
\]
The proof of the following result follows inmediately after from Proposition  \ref{prop:degreeproposition} and Lemma \ref{lem:usefulparaboliclemma}.
\begin{thm}\label{thm:existencemultibumps}
There exists $\tilde{\varphi}_*\in H^{1}([-n,n])$,  $T>0$ and $\varepsilon>0$  such that for all $L\in\mathbb{N}$  
\[
\mathrm{deg}(F, (B_\varepsilon(\varphi_*))^{L+1}\times  (\{l\in\mathbb{N}\colon l\geq T\}^{L},0)=(-1)^L
\]
In particular, for any sequence $\mathbf{l}=\{l_j\}_{1\leq j\leq L}\subset (\{l\in\mathbb{N}\colon l \geq T\}^{L}$ there exists $\boldsymbol{\varphi}(\mathbf{l})=\{\varphi_j (\mathbf{l})\}_{1\leq j\leq L+1}\subset (H^1([-n,n]))^{L+1}$ such that 
\[
F(\boldsymbol{\varphi}(\mathbf{l}),\mathbf{l})=0.
\]
\end{thm}

Theorem \ref{thm:existencemultibumps} shows the truth of the first item in Theorem \ref{thm:maintheorem} for sequences $\sigma\in \{0,1\}^{\mathbb{Z}}$ with finite, but arbitrarily large number of nonzero entries. For the time interval $T_{**}$ in the definition of \eqref{eq:multibumpmap} does not depend on $L$,  the existence of solutions $r_\sigma$ such that $\sigma$ has infinitely many non-zero entries follows by a standard diagonal argument in the $C^1_{\mathrm{loc}}$ topology. In order to deduce the second item, namely the existence of infinitely many periodic orbits, we define the functional 
\begin{equation}
\begin{aligned}
F_{\mathrm{per}}: (B_\varepsilon(\varphi_*))^{L} \times (\{l\in\mathbb{N}\colon l\geq T_{**}\})^{L}& \longrightarrow (H^1([-n,n]))^{L}\\
(\varphi_1,\dots,\varphi_{L+1},l_1,\dots,l_L)& \longmapsto (F_1,\dots,F_{L})
\end{aligned}
\end{equation}
with periodic boundary conditions
\[
\begin{split}
F_1=&\nabla\mathcal{\widetilde{A}}_G  +\left(\partial_r S^+(\varphi_1(n),n)-v^+(l_1,\varphi_1(n),\varphi_2(-n))\right)\mathtt{ev}_n\\
&+\left( v^-(l_{L},\varphi_{L}(n),\varphi_{1}(-n))-\partial_r S^-((\varphi_1)(-n),-n)\right)\mathtt{ev}_{-n}\\
F_{L}=&\nabla\mathcal{\widetilde{A}}_G+ \left(v^-(l_{L},\varphi_{L}(n),\varphi_{L+1}(-n))- \partial_r S^-(\varphi_{L+1}(-n),-n)\right)\mathtt{ev}_{n}\\
&+\left( v^+(L,\varphi_{L}(n),\varphi_1(-n))-\partial_r S^+((\varphi_L)(-n),-n)\right)\mathtt{ev}_{-n}
\end{split}
\]
and such that for $2\leq j\leq L$ (this set is empty for $L=1$) $F_j$ has the same expression as in in the non periodic case. The proof of Theorem \ref{thm:maintheorem} is complete.

\subsection{Proof of Theorem \ref{thm:mainthmfinalmotions}}
From the proof of Theorem \ref{thm:maintheorem} it follows that
\[
X^+\cap Y^-\neq \emptyset
\]
for any possible combination of $X^+\in\{P^+,B^+,OS^+\}$ and $Y^-\in\{P^-,B^-,OS^-\}$. We now show that $H^+\cap P^-\neq \emptyset$ (the proof for the other combinations being similar). The following result is implied by the second part of  Lemma \ref{lem:lambdalemmaoriginalcoord}.

\begin{lem}\label{lem:usefulhyperboliclemma}
There exists $R$  large enough and $\varepsilon_0$ such that for all $R\geq R_0$ and all $0\leq\varepsilon\leq\varepsilon_0$  there exists a unique orbit $\hat{r}(s;,R_0,\delta)$ of \eqref{eq:Hamiltonian} for which 
\[
\hat{r}(0)=R_0,\qquad\qquad\qquad\dot{\hat{r}}(0)=\partial_r S^+ (R_0,0)+\varepsilon.
\]
Moreover, $\hat{r}(s;,R_0,\delta)$ is defined for all $s\geq 0$ and satisfies 
\[
\lim_{s\to\infty} \hat{r}(s;R_0,\varepsilon)=\infty\qquad\qquad\qquad\lim_{s\to\infty} \dot{\hat{r}}(s;R_0,\varepsilon)>0
\]
\end{lem}
Given $R_0\geq R$ and $\varepsilon>0$ we denote by 
\[
v^+_{\mathrm{hyp}}(R_0,\varepsilon)=\dot{\hat{r}}(0;R_0,\varepsilon)
\]
where $\hat{r}(s;R_0,\varepsilon)$ is the orbit segment found in Lemma \ref{lem:usefulhyperboliclemma}. Fix $0\leq\varepsilon\leq \varepsilon_0$. The fact that $H^+\cap P^-\neq \emptyset$ follows from the fact that the functional 
\[
\begin{aligned}
F_{\mathrm{hyp}}: (B_\varepsilon(\varphi_*))^{L+1} \times (\{l\in\mathbb{N}\colon l\geq T_{**}\})^{L}& \longrightarrow (H^1([-n,n]))^{L+1}\\
(\varphi_1,\dots,\varphi_{L+1},l_1,\dots,l_L)& \longmapsto (F_{1,\mathrm{hyp}},\dots,F_{L+1,\mathrm{hyp}})
\end{aligned}
\]
where the maps $F_{j,\mathrm{hyp}}$, $1\leq j\leq L+1$ are given by 
\[
\begin{split}
F_{1,\mathrm{hyp}}=&\nabla\mathcal{\widetilde{A}}_G  +\left(\partial_r S^+(\varphi_1(n),n)-v^+(l_1,\varphi_1(n),\varphi_2(-n))\right)\mathtt{ev}_n\\
F_{L+1,\mathrm{hyp}}=&\nabla\mathcal{\widetilde{A}}_G+ \left(v^-(l_{L},\varphi_{L}(n),\varphi_{L+1}(-n))- \partial_r S^-(\varphi_{L+1}(-n),-n)\right)\mathtt{ev}_{n}\\
&+ \left(\partial_r S^+(\varphi_{L+1}(n),n)-v^+_{\mathrm{hyp}}(\varphi_{L+1}(n),\varepsilon)\right)\mathtt{ev}_n
\end{split}
\]
and for $2\leq j\leq L$ (this set is empty for $L=1$)
\[
\begin{split}
F_{j,\mathrm{hyp}}=&\nabla\mathcal{\widetilde{A}}_Gi+ \left(\partial_r S^+(\varphi_j(n),n)-v^+(l_{j},\varphi_{j}(n),\varphi_{j+1}(-n))\right)\mathtt{ev}_{n}\\
&+\left( v^-(l_{j-1},\varphi_{j-1}(n),\varphi_j(-n))-\partial_r S^-((\varphi_j)(-n),-n)\right)\mathtt{ev}_{-n}.
\end{split}
\]
In order to prove that $P^+\cap OS^-\neq\emptyset$ we take $L\to\infty$ and argue as in the proof of Theorem \ref{thm:maintheorem}. In order to show that $H^+\cap B^-$ we impose periodic boundary conditions. The proof of Theorem \ref{thm:mainthmfinalmotions} is complete.

\appendix

\section{Proof of of the technical claims in Lemma \ref{lem:asymptoticactions}}\label{sec:technicallemmas}

We first prove the following result, which will be needed for the proof of Lemma \ref{lem:differencegeneratingfunctions}.
\begin{lem}
Let $S_0(r;G)$ be the generating function defined in Lemma \ref{lem:hamiltonjacobi2bp}. Then, for any $G,G_*\in\mathbb{R}$ we have that 
\[
|S_0(r;G)-S_0(r;G_*)|\lesssim \frac{|G^2-G_*^2|}{r^{1/2}}.
\]
\end{lem}

\begin{proof}
 Denote by $\tilde{u}(r)$ the unique function such that, for all $y<0$, the $y$ component of the parametrization \eqref{eq:homoclinicparametrization2BP} is given by $y_h(\tilde{u}(r))$. Writing $\Delta S(r;G,G_*)=S_0(r;G)-S_0(r;G_*)$ we observe that $\Delta S(r;G,G_*)$ satisfies
\[
y_h(\tilde{u}(r)) \partial_r \Delta S+ \frac{G^2-G_*^2}{2 r^2}+\frac{(\partial_r \Delta S)^2}{2}=0 
\]
Using now that $y_h(\tilde{u}(r))\sim r^{-1/2}$ for large $r\gg1$ we obtain that 
\[
\partial_r \Delta S \sim \frac{G^2-G_*^2}{ r^{3/2}}+ \mathcal{O}(r^{1/2} (\partial_r \Delta S)^2)
\]
so 
\[
|\Delta S(r;G,G_*))|\lesssim \frac{|G^2-G_*^2|}{r^{1/2}}
\]
as was to be shown.
\end{proof}

The claims in Lemma \ref{lem:asymptoticactions} follow from the following result.

\begin{lem}\label{lem:differencegeneratingfunctions}
Suppose that for $G\in [-G_*,G_*]$ there exists an orbit $r(s;G):\mathbb{R}\to\mathbb{R}_+$ of the Hamiltonian $H_G$ in\eqref{eq:oneparameterHamiltonian} which is homoclinic to $\gamma_\infty$ and, for some $G_0\in [-G_*,G_*]$ satisfies
\[
|r(s;G)-r_0(s,G_0)|\lesssim s^{1/2}\qquad\qquad \text{as}\qquad s\to\pm\infty.
\]
Then,
\[
|\partial_r S^+(r(s;G),G)-\partial_r S_0(r_0(s;G_0),G_0)|\lesssim s^{-5/6} \qquad\qquad \text{as}\qquad s\to\pm\infty,
\]
and, in particular
\[
\lim_{s\to\pm\infty} S^\pm (r(s;G),s;G)-S^0(r_0(s;G_0),G_0)=0.
\]
\end{lem}

\begin{proof}
We write
\[
\begin{split}
\partial_r S^+(r(s;G),G)-\partial_r S_0(r_0(s;G_0),G_0)=&\left(\partial_r S^+(r(s;G),G)-\partial_r S_0(r(s;G),G)\right)\\
&+\left(\partial_r S^0(r(s;G),G)-\partial_r S_0(r_0(s;G_0),G)\right)\\
&+\left(\partial_r S^0(r_0(s;G_0),G)-\partial_r S_0(r_0(s;G_0),G_0)\right)\\
=& E_1+E_2+E_3.
\end{split}
\]
On one hand, it follows from the last item in  Lemma \ref{prop:generatingfunctions} that as $s\to\pm\infty$
\[
|E_1|\lesssim r^{-5/2}(s;G)\lesssim r_0^{-5/2}(s;G)\lesssim s^{-5/3}.
\]
On the other hand, it follows from the mean value theorem, the definition of $S^0(r;G)$ and the hypothesis in the statement of the lemma that as $s\to\pm\infty$
\[
|E_2|\lesssim \sup_{r\in I} |\partial_{rr}^2 S^0(r;G)| |r(s;G)-r_0(s,G_0)|\lesssim r_0^{-2}(s;G_0) |r(s;G)-r_0(s,G_0)|\lesssim s^{-5/6}
\]
for $I=\{r\in\mathbb{R}_+\colon r=\lambda r_0(s;G_0)+(1-\lambda) r(s;G),\ \lambda\in[0,1]\}$. Also, from Lemma \ref{lem:hamiltonjacobi2bp} we deduce that 
\[
|E_3|\lesssim r^{-3/2}\lesssim s^{-1}.
\]
The proof of the first item follows combining the estimates for $E_1,E_2,E_3$ and integrating. The second part follows from the obtained estimate and straightforward computations.
\end{proof}

\bibliographystyle{alpha}
\bibliography{biblioMelnikov}

\end{document}